\documentclass[sn-mathphys-num]{sn-jnl}

\usepackage{graphicx}%
\usepackage{multirow}%
\usepackage{amsmath,amssymb,amsfonts}%
\usepackage{amsthm}%
\usepackage{mathrsfs}%
\usepackage[title]{appendix}%
\usepackage{xcolor}%
\usepackage{textcomp}%
\usepackage{manyfoot}%
\usepackage{booktabs}%
\usepackage{algorithm}%
\usepackage{algorithmicx}%
\usepackage{algpseudocode}%
\usepackage{listings}%
\usepackage{bm}
\usepackage{mathtools}
\usepackage{tikz}
\usetikzlibrary{shapes}
\usetikzlibrary{positioning}
\usepackage{rotating}
\usepackage{xurl}

\newtheorem{theorem}{Theorem}
\newtheorem{lemma}{Lemma}
\newtheorem{assumption}{Assumption}
\newtheorem{proposition}{Proposition}

\DeclareMathOperator*{\argmax}{arg\,max}

\begin{document}

\title[Article Title]{Global and Robust Optimization for Non-Convex Quadratic Programs}

\author[]{\fnm{Asimina} \sur{Marousi}}

\author*[]{\fnm{Vassilis} \sur{M. Charitopoulos}}\email{v.charitopoulos@ucl.ac.uk}

\affil[]{\orgdiv{Department of Chemical Engineering, Sargent Centre for Process Systems Engineering}, \orgname{University College London}, \orgaddress{\street{Torrington Place}, \city{London}, \postcode{WC1E 7JE},  \country{UK}}}


\abstract{This paper presents a novel algorithm integrating global and robust optimization methods to solve continuous non-convex quadratic problems under convex uncertainty sets. The proposed Robust spatial branch-and-bound (RsBB) algorithm combines the principles of spatial branch-and-bound (sBB) with robust cutting planes. We apply the RsBB algorithm to quadratically constrained quadratic programming (QCQP)  problems, utilizing McCormick envelopes to obtain convex lower bounds. The performance of the RsBB algorithm is compared with state-of-the-art methods that rely on global solvers. As computational test bed for our proposed approach we focus on pooling problems under different types and sizes of uncertainty sets. The findings of our work highlight the efficiency of the RsBB algorithm in terms of computational time and optimality convergence and provide insights to the advantages of combining robustness and optimality search.}




\maketitle

\section{Introduction}\label{sec1}

One of the fundamental assumptions when constructing mathematical programming models is the degree of uncertainty in the input parameter data. If uncertainty is neglected, then a deterministic optimization problem is generated, while for the case where uncertainty is considered to some degree, there are various methods to model uncertainty. There are two key drivers for the selection of the appropriate optimization method: statistical data availability and desired degree of robustness. Robust optimization has emerged as a leading method for problems with limited uncertain parameter data and requirements for a risk-averse solution. In a robust optimization context, solutions that lead to attainable decisions under all parameter realizations are referred to as robust feasible, while solutions that also leads to the best objective value are called robust optimal. A general form of the robust optimization problem is given by \eqref{eq:GRO}. 

\begin{equation}\tag{$P_{rob}$}\label{eq:GRO}
\begin{aligned}
&\underset{\bm x\in \mathcal{X}}{\min}~f(\bm x)\\
&\text{s.t.}~ g_i(\bm x,\bm u)\leq 0,~\forall \bm u \in \mathcal{U},\quad i=1,\cdots,m
\end{aligned}
\end{equation}

where $\bm x\in \mathcal{X}\subset \mathbb{R}^{n_x}$ denotes continuous decisions, and $u \in \mathcal{U} \subset \mathbb{R}^{n_u} $ is the vector of uncertain parameters that resides within an uncertainty set $\mathcal{U}$, $f:\mathcal{X}\rightarrow\mathbb{R},~ g_i:\mathcal{X}\rightarrow\mathbb{R}$. Depending on the properties of $f$ and $g_i$ with respect to the optimization variables $\bm x$, of the  uncertainty set $\mathcal{U}$ and $g_i$ with respect to the uncertain parameters $\bm u$ either reformulation or sampling approaches can be employed for the solution of \eqref{eq:GRO}.

State-of-the-art robust optimization algorithms for non-convex problems typically rely on global solvers to obtain robust optimal solutions. The prevailing approach involves deriving the dual reformulation and utilizing a global solver, either directly or adaptively. However, the dual reformulation can increase the problem complexity hence impeding the convergence of the global solver. In the case of challenging non-convex problems convergence to a global solution could not be achieved within a predefined optimality tolerance. Nevertheless, the obtained local solution would be feasible for the robust problem. An alternative approach is based on an iterative robust cutting plane algorithm. In this case, significant computational time can be spent searching for a global solution that may be deemed robust infeasible. 

\textbf{Contributions of this work:} Our research hypothesis is that integrated exploration of global and robust optimality can yield computational benefits. To this end, in this article we introduce a novel algorithm that conducts a concurrent global optimality and robustness search for continuous non-convex problems. The proposed approach integrates spatial branch-and-bound with robust cutting plane notions. The key idea is that, while exploring the branch-and-bound tree, we assess the robustness of nodes entailing the best-found solutions. We illustrate the performance and benefits of our proposed approach through benchmark Quadratically Constrained Quadratic Programs (QCQPs) of pooling problems. At each node, the non-convex problem is solved via a local solver. If the computed solution is as good as the best-found so far, an infeasibility test is performed to evaluate the robustness of the obtained solution. If not, then the corresponding cutting planes are added both to the original non-convex and relaxed convex problems. The algorithm proceeds to the next step once no more violations are detected. The convex problem is solved next, and the solution of this problem is used to decide the most promising variable for branching into two child nodes. With the use of appropriate fathoming criteria, the tree nodes are exhausted, and a robust optimal solution is obtained. 

Section \ref{LR} provides a literature review on state-of-the-art nonlinear robust optimization and semi-infinite programming (SIP) methods and their applications, along with a brief overview of the pooling problem which serves as our case study. The outline of the robust cutting planes algorithm and the details for proposed Robust spatial branch-and-bound algorithm are introduced in Section \ref{meth}. In Section \ref{prob}, we present the pooling problem under uncertainty, and derive the corresponding formulations for the dual reformulation and RsBB solution methods. In Section \ref{results}, we evaluate the performance of the RsBB algorithm compared to the state-of-the-art for benchmark pooling problems under box, ellipsoidal and polyhedral sets of varying sizes.

\section{Literature review}\label{LR}
Robust optimization (RO) is developing into one of the prevailing risk-averse methods to study problems under uncertainty in the optimization literature \citep{Zhang2022,Gabrel2014}. It is largely preferred for applications where there is no or limited statistical data for the uncertain parameters, or when dealing with hard feasibility constraints. In RO, the modeler must decide an appropriate, in size and shape, uncertainty set to characterize the uncertain parameters. The RO algorithms provide an optimal solution that corresponds to the worst-case uncertainty realization and at the same time guarantee that the solution is feasible given any parameter value from the selected uncertainty set. The main challenge of RO is to find the appropriate way to transform the uncertain optimization problem into a tractable deterministic formulation. The most prevailing methods in that direction, are dual reformulation \citep{Bental98, Gorissesn2015,Li2011} and robust cutting sets \citep{Mutapcic2009,Isenberg2021}. The first comparison of the two methods can be traced to \citet{Fischetti2012} who applied both methods in linear and mixed integer problems. Based on the computational experiments, cutting planes were found to be more suitable for the linear programs, while the dual reformulation performed better for the mixed integer programs. \citet{Bertsimas2016} quantified the performance of the two methods for different variants of the cutting planes, as well as that of a hybrid approach. In contrast to previous results in the literature, for the ellipsoidal uncertainty sets, the dual reformulation dominated in the linear programs and cutting planes for mixed-integer instances. For polyhedral uncertainty sets neither method was deemed to be superior.

By definition, robust optimization problems belong to the general class of SIP\citep{Stein2012}. The solution approaches for SIP problems can be divided into two general categories: discretization and overestimation methods. The discretization method as initially proposed by \citet{Blankenship1976} performed a discretization of the infinite variable set while iteratively solving the upper and lower-level problems  similar to the robust cutting set algorithm. For continuous functions and compact host sets, the sequence of upper-level solution points approaches a solution of the SIP at the limit.  To address the problem of possible infeasibility in the upper-level problem (ULP), \citet{Mitsos2011} proposed an iterative global optimization method that involved constructing a decreasing sequence of right-hand side constraint relaxations. The feasibility of the obtained solution was evaluated solving the respective lower-level problem (LLP). The proposed algorithm was proven to finitely terminate in an $\epsilon-$optimal solution.  This method was later generalized to address general SIP problems \cite{Mitsos2015}. Apart from the feasibility based discretization methods, the reduction-based adaptations rely on locally solving the ULP to acquire approximate solutions which are updated by the solutions of the LLP which are solved by a global solver\citep{LopezMarco2007,Seidel2022}. Overestimation methods entail interval analysis \citep{Bhattacharjee2005,Bhattacharjee2005math, Mitsos2008} and relaxation methods \citep{Floudas2007,Stein2012adaptive} to generate convergent sequences of upper and lower bounds to retrieve  optimal solutions for SIP problems.

The study of non-linear problems by the RO\citep{Leyffer2020, hung2024solution} and SIP \citep{Djelassi2021} literature has been largely developed interdependently despite their commonalities. For nonlinear convex robust problems under convex uncertainty sets both dual reformulation and robust cutting set methods can be applied. \citet{Diehl2005} highlight the computational challenge in solving nonlinear robust optimization problems under general norms and propose an approximate reformulation that employs a linearization of the uncertainty set. For nonlinear convex problems convex analysis may also be applied to retrieve tractable reformulations for specific classes of uncertainty sets \citep{BentalNRO, Gorissesn2015}. \citet{Zhang2007} presented a rigorous method to choose safety margins for uncertain parameters in nonlinear and non-convex problems via local linearization around a nominal uncertainty value. Instead of using the nominal uncertainty point, \citet{Li2018nro} proposed an algorithm based on outer approximation over sampled uncertainty points and tested their methodology in non-convex process systems engineering problems. \citet{Bertsimas2009} performed a neighborhood search using a robust local descent method to avoid locally infeasible regions. Their methodology obtained local robust solutions for non-convex and simulation-based problems. \citet{Houska2012} used a sequential bilevel approach for non-convex min-max problems resulting in less conservative solutions compared to the state-of-the-art.

In various applications, finding a local robust solution may not suffice, necessitating thus the need for coupling global and robust optimization. Polynomial max-min and min-max problems are represented as general semi-definite programming(SDP) problems and their corresponding robust optimal values can be approximated solving a hierarchy of SDP relaxations \citep{Lasserre2006, Lasserre2011}. \citet{Li2011_gro} studied the scheduling of crude oil operations under demand uncertainty. The authors derived the deterministic robust counterpart of the uncertain constraints using a tailored branch-and-bound (BB) algorithm. \citet{Wiebe2019} addressed the pooling problem under uncertainty comparing dual reformulation and cutting plane methods using a global solver. For the cutting planes method they relied on using a global solver and gradually decreasing the solver tolerance. Cutting planes outperformed the dual reformulation approach in terms of \% of instances solved within an imposed time limit, since the latter increased the problem complexity for ellipsoidal and polyhedral uncertainty sets.  \citet{Isenberg2021} extending the work of \citet{Mutapcic2009}, and proposed a robust cutting set algorithm being able to certify robust solutions for non-convex problems entailing a large number of equality constraints. The quality of the final solution, i.e. being robust feasible or robust optimal, is solver dependent. On the same year, \citet{Zhang2021} employed an enhanced normalized multi-parametric disaggregation technique (ENMDT) and optimality-based bound tightening to solve the problem of refinery-wide  planning operations under uncertainties. \citet{Carrizosa2021} introduced a BB algorithm with interval arithmetic to solve benchmark unconstrained non-convex optimization problems of low-dimensionality under robust uncertainty. Most robust optimization methods require the convexity of the uncertainty set or the lower-level problem, however \citet{Kuchlbauer2022} proposed an adaptive bundle method for MINLP problems under robust uncertainty that allows the lower-level problem to be non-convex.

We derive our QCQP case study from the process systems engineering literature as the planning of pooling problems. The pooling problem  belongs to the general class of non-convex quadratic problems and is proven to be NP-hard even for a small number of variables \citep{Alfaki2013}. Historically, the most common solution approach to address non-convex NLPs, such as the pooling problem, is spatial branch-and-bound (sBB) \citep{Tuy1998}.  In sBB, in contrast to classic branch-and-bound, branching is performed on continuous variables. The original non-convex problem constitutes the upper bound and the equivalent relaxed linear programming (LP) problem the lower bound. \citet{Foulds1992} used convex relaxations to approximate the bilinear terms within a branch-and-bound framework. \citet{Quesada1995} integrated the McCormick envelopes \citep{McCormick1976} with  Reformulation Linearization Technique (RLT) constraints \citep{Sherali1992} in an sBB framework selecting the branching variable as the one with the worst-approximation error. Using RLT constraints result in tightening the convex relaxations of the bilinear terms without affecting the feasible region.  Later on, \citet{Alfaki2013} used both convex and concave envelopes following a novel branching strategy.  For high-dimensional problems, the number of convex relaxation constraints increases drastically increasing the complexity of the problem. To that end, \citet{Audet2004} proposed a branch-and-cut algorithm where only the violating linearization constraints are added as cuts, instead of all convex and concave envelopes. Every generated cut is applied to all tree nodes. \citet{Gounaris2009}, replaced the non-convex pooling problem  with a piecewise linear relaxation resulting in an MILP problem.  \citet{Dey2015} studied the impact of using piecewise linearization for the linear envelopes, they also introduced pooling benchmark problems of increased dimensionality and complexity. \citet{Ceccon2022} introduced the open-source MIQCQP solver GALINI that is based on a branch-and-cut algorithm using convex linear relaxations and mixed-integer restrictions for the pooling problems. The performance of the MILP pooling formulation can be improved using appropriate tightening methods and reduce computational time \cite{Chen2024}. In this work following the results on the pooling problem literature, we employ an sBB approach under McCormick relaxations as the selected global optimization methodology as it enables the integration of of robust cutting planes.

\section{Methods}\label{meth}
The proposed Robust spatial branch-and-bound (RsBB) integrates the robust cutting set algorithm within a spatial branch-and-bound global optimization algorithm in search of an optimal robust solution. In this section, the robust cutting set algorithm \citep{Mutapcic2009} will be outlined with a theoretical study on the convergence of the integrated robust cutting planes and BB algorithms. In the following subsection, the proposed RsBB algorithm is introduced. The assumptions for our methodology are the following.
\begin{assumption}\label{as:bounded}
The domain of $\bm x$, $\mathcal{X}\subset \mathbb{R}^{n_x}$ is compact such that $\mathcal{X}=\{\bm x |~  \underline{\bm x}\leq \bm x \leq  \overline{ \bm x}\}$. 
\end{assumption}
Assumption \ref{as:bounded} ensures compactness of the feasible set which is a requirement for the finite convergence of the BB methods \citep{Belotti2013}.

\begin{assumption}\label{as:lips}
    The objective and constraint functions $f,~g_i: \mathcal{X}\rightarrow \mathbb{R}$ are bilinear or linear on $\bm x$.
\end{assumption}

Assumption \ref{as:lips} ensures the differentiability and continuity of the functions. In combination with Assumption \ref{as:bounded} Lipschitz continuity of the problem function is inferred\citep{Stein2024} which is a requirement for the finite convergence of the robust cutting set algorithm \citep{Mutapcic2009}. Moreover, this assumption supports the satisfaction of standard regularity conditions that underlie the convergence of local nonlinear solvers.

\begin{assumption}\label{as:onu}
   The uncertainty set $\mathcal{U}$ is convex and generates a non-empty robust feasible region.
\end{assumption}

If Assumption \ref{as:onu} does not hold, and $\mathcal{U}$ generates an empty feasible region, then the robust problem is infeasible hence no decision $\bm x \in \mathcal{X}$ can satisfy the uncertainty constraints for all the uncertain parameters $\bm u \in \mathcal{U}$ \citep{Bental09,Bertsimas04}.

\begin{assumption}\label{as:gaff}
    The uncertain constraints $g_i$ are affine in $\bm u$.
\end{assumption}

Assumptions \ref{as:onu} and \ref{as:gaff} guarantee that the solution of the LLP problems is exact hence the robust cutting plane algorithm terminates in a finite number of iterations.

\subsection{Robust cutting set algorithm}
Consider the deterministic problem \eqref{eq:nom} for the nominal uncertainty value $\bm u_{nom}$.

\begin{equation}\tag{$P_{nom}$}\label{eq:nom}
\begin{aligned}
&\underset{\bm x\in \mathcal{X}}{\min}~f(\bm x)\\
&\text{s.t. }~ g_i(\bm x,\bm u_{nom})\leq 0, \quad i=1,\cdots,m
\end{aligned}
\end{equation}

Following \citet{Mutapcic2009}, there exists a subset $ \hat{\mathcal{U}_i}=\{ u_{i,0},\cdots, u_{i,K_i}\}\subseteq \mathcal{U},~ i=1,\cdots,m$ involving $K_i$ finite uncertainty samples or realizations for each uncertain constraint and the corresponding problem \ref{eq:samp}.   
\begin{equation}\tag{$P_{samp}$}\label{eq:samp}
\begin{aligned}
&\underset{\bm x\in \mathcal{X}}{\min}~f(\bm x)\\
&\text{s.t.}~ g_i(\bm x,\bm u)\leq 0,~ \bm u \in \hat{\mathcal{U}_i}, \quad i=1,\cdots,m
\end{aligned}
\end{equation}

Let $\bm u_{nom}=\bm u_0$, the following uncertainty sample can be determined via solving the upper-level minimization problem (ULP) \ref{eq:nom}, and evaluating the robustness of the solution $\bm x^*_{nom}$ in the lower-level maximization problem (LLP) \eqref{eq:maxg} $\forall i$.
\begin{equation}\tag{$T_i$}\label{eq:maxg}
\begin{aligned}
&\underset{\bm u}{\max}~g_i(\bm x^*, \bm u)\\
&\text{s.t. }\,\bm u\in \mathcal{U}
\end{aligned}
\end{equation}
The robust cutting set Algorithm \ref{alg:RCP} terminates in a robust optimal solution after a finite number of iterations $K$ given that the exact solution of problem \eqref{eq:maxg} can be obtained. The exactness of the solution of problems \eqref{eq:maxg} can be guaranteed for $g_i$ and $\mathcal{U}$ satisfying any of the following properties \citep{Mutapcic2009}:
\begin{itemize}
    \item $\mathcal{U}$ is a finite set.
    \item $g_i$ is concave in $\bm u$, for each $\bm x$, and $\mathcal{U}$ is a polyhedron defined as the convex hull of a finite set of modest size.
    \item $g_i$ is monotone for each $\bm u \in \mathcal{U}$ and for each $\bm x$, and $\mathcal{U}$ is a box set around the nominal point $\bm u_0$.
    \item $g_i$ is affine in $\bm u$ and $\mathcal{U}$ is an ellipsoid around the nominal point $\bm u_0$.
        \item $g_i$ is affine in $\bm u$ and $\mathcal{U}$ is a polyhedron around the nominal point $\bm u_0$.
\end{itemize}

If $g_i$ is quadratic and $\mathcal{U}$ is defined by a quadratic function, an S-procedure can be employed satisfying the exact solution of the LLP \citep{Mutapcic2009}. For the cases where the above conditions are not satisfied, approximate solution methods can be employed such as gradient-descent methods and second-order methods, resulting in early termination of the robust cutting plane algorithm.  For non-convex LLP problems approximate solution approaches have been proposed relying on the restricted right hand side methodology to gradually reduce the approximation error based on SIP literature \cite{Mitsos2009,Mitsos2015,Harwood2021}.

In this study, we focus on the case where $g_i$ is affine in $\bm u$ and we are evaluating box, polyhedral and ellipsoidal sets, hence exactness of the solution is satisfied. For a specified feasibility tolerance $\delta>0$, if a solution $\bm u^*$ is detected, such that $g_i(\bm x^*,\bm u^*)>\delta$, then it is added on subset $\hat{\mathcal{U}_i}$.
Let $\hat{G}(\bm x)$ be the worst-case constraint function 
\begin{equation*}
    \hat{G_i}(\bm x)=\max \{g_i(\bm x,\bm u_{i,0}),\cdots,g_i(\bm x,\bm u_{i,K_i})\}~ i=1,\cdots,m
\end{equation*}

it is true that 
\begin{equation*}
    g_i(\bm x, \bm u_{i,0})\leq\hat{G}(\bm x)\leq G(\bm x)
\end{equation*}

where $G(x)=\underset{\bm u\in \mathcal{U}}{\max }\, g_i(\bm x,\bm u)$. 
\begin{proposition}\label{pr:bound}
An optimal solution of the sampled problem provides a lower bound for the robust problem.
\end{proposition}

\begin{proof}
    Let $\mathcal{F}_{rob},~\mathcal{F}_{samp},~\mathcal{F}_{nom}$ the feasible sets and $z^*_{rob},~z^*_{samp},~z^*_{nom}$ the optimal values for the robust, sampled and nominal problems. Then,
    $\mathcal{F}_{nom}:=\{\bm x\in \mathcal{X}, \bm u \in \hat{\mathcal{U}}_0: g_i(\bm x,\bm u)\leq0\}$, $\mathcal{F}_{samp}:=\{\bm x\in \mathcal{X}, \bm u \in \hat{\mathcal{U}_i}: g_i(\bm x,\bm u)\leq0\}$ and $\mathcal{F}_{rob}:=\{\bm x\in \mathcal{X}, \bm u \in\mathcal{U}: g_i(\bm x,\bm u)\leq0\}$. By definition $\hat{\mathcal{U}}_0\subseteq\hat{\mathcal{U}}\subseteq\mathcal{U}$. Hence, $\mathcal{F}_{rob}\subseteq \mathcal{F}_{samp} \subseteq \mathcal{F}_{nom}$ and $z^*_{rob}\geq z^*_{samp}\geq z^*_{nom}$.
\end{proof}

\begin{proposition}\label{pro:opt}
    If an optimal solution of the sampled problem is feasible for the robust problem, then it is also optimal for the robust problem.
\end{proposition}

\begin{proof}
Let $\bm x^*_{samp}\in \mathcal{F}_{samp}$ be an optimal solution for the sampled problem. From Proposition \ref{pr:bound} $z^*_{rob}\geq z^*_{samp}$. If $G(\bm x^*_{samp})\leq0$ then $\bm x^*_{samp}\in \mathcal{F}_{rob}$ as the sampled solution belongs to the robust feasible region. Hence  $z^*_{rob}= z^*_{samp}$ at $\bm x^*_{samp}$.     
\end{proof}

An outline of the robust cutting set algorithm is provided in Algorithm \ref{alg:RCP}. We note that \citet{Isenberg2021} have implemented a generalized cutting set method in the PyROS solver enabling the solution of non-convex problems as well.

\begin{algorithm}
\caption{Robust Cutting Set Algorithm}
\label{alg:RCP}
\begin{algorithmic}[1] 
\Require Feasibility tolerance $\delta>0$
    \State Solve \eqref{eq:nom} $\rightarrow \bm x^*$
    \For{$i=1,\cdots,m$}
    \State Solve \eqref{eq:maxg} $\rightarrow \bm u^*$ 
    \If {$g_i(\bm u^*, \bm x^*) > \delta$}
        \State Append the robust set $\hat{\mathcal{U}_i}\leftarrow \hat{\mathcal{U}_i} \cup \{\bm u^*\}$
\EndIf
\EndFor
\If {$\underset{i=1,\cdots,m}{\max} g_i(\bm u^*, \bm x^*) > \delta$}
            \State Solve \eqref{eq:samp} $\rightarrow \bm x^*$
            \State Go to Step 2
    \Else
        \State Go to Step 12
    \EndIf

    \State \textbf{Terminate:} Robust solution obtained
\end{algorithmic}
\end{algorithm}

\subsection{Robust spatial branch-and-bound}
The aim of this work is to solve problems with non-convex bilinear functions $f,~g_i$ in $\bm x$ under robust uncertainty using an sBB algorithm via McCormick relaxations. We define the auxiliary variables for the bilinear $\bm y\in \mathcal{Y}\subset \mathbb{R}^{n_b\times n_b}$ such that $y_{bj}=x_bx_j, ~ b,j=1,\cdots, n_b$. Hence, we derive the augmented representation of \eqref{eq:GRO} using the auxiliary variables such that:

\begin{equation}\tag{$P'_{rob}$}\label{eq:auxGRO}
\begin{aligned}
&\underset{\bm x, \bm y}{\min}~\tilde{f}(\bm x, \bm y)\\
&\text{s.t.}~ \tilde{g_i}(\bm x,\bm y,\bm u)\leq 0,~\forall \bm u \in \mathcal{U},\quad i=1,\cdots,m\\
&\ y_{bj}=x_bx_j, ~b,j=1,\cdots,n_b\\
&\ \bm x \in \mathcal{X}, \bm y \in \mathcal{Y}
\end{aligned}
\end{equation}

where $\tilde{f},~\tilde{g}_i:\mathcal{X}\times \mathcal{Y}\rightarrow\mathbb{R}$ denote the augmented objective function and constraints when the auxiliary variables $\bm y$ are employed. Note that problems \eqref{eq:GRO} and \eqref{eq:auxGRO} are equivalent and both non-convex. To derive the convex relaxation of \eqref{eq:auxGRO} the equality constraints for the auxiliary variables are relaxed using McCormick envelopes\citep{McCormick1976}:

\begin{equation}
   \mu_{bj}(\bm x, \bm y)\leq 0,~b,j=1,\cdots, n_b 
\end{equation}
where
\begin{equation}\label{eq:env}
    \mu_{bj}(\bm x, \bm y):=\begin{cases}
     \underline{x_b}x_j+x_b\underline{x_j} -\underline{x_b}\underline{x_j}-   y_{bj} \\
     y_{bj}-\overline{x_b}x_j-x_b\underline{x_j}+\overline{x_b}\underline{x_j}\\
     \overline{x_b}x_j +x_b\overline{x_j}-\overline{x_b}\overline{x_j}-y_{bj}\\
     y_{bj}-\underline{x_b}x_j-x_b\overline{x_j}+\underline{x_b}\overline{x_j}\
    \end{cases} \quad b,j=1,\cdots, n_b
\end{equation}

Then the relaxed convex nominal problem \eqref{eq:nom_cv} and sampled problem \eqref{eq:samp_cv} are formulated as follows.

\begin{equation}\tag{$\tilde{P}_{nom}$}\label{eq:nom_cv}
\begin{aligned}
&\underset{\bm x, \bm y}{\min}\tilde{f}(\bm x, \bm y)\\
&\text{s.t.}~ \tilde{g}_i(\bm x,\bm y,\bm u_{nom})\leq 0 , ~i=1,\cdots,m\\
& \qquad \mu_{bj}(\bm x,\bm y)\leq 0, ~b,j=1,\cdots, n_b\\
&\bm x \in \mathcal{X}, ~\bm y \in \mathcal{Y}
\end{aligned}
\end{equation}

\begin{equation}\tag{$\tilde{P}_{samp}$}\label{eq:samp_cv}
\begin{aligned}
&\underset{\bm x, \bm y}{\min}\tilde{f}(\bm x, \bm y)\\
&\text{s.t.}~ \tilde{g}_i(\bm x,\bm y,\bm u)\leq 0 ,~ \bm u \in \hat{\mathcal{U}_i} ~i=1,\cdots,m\\
& \qquad \mu_{bj}(\bm x,\bm y)\leq 0,  ~ ~b,j=1,\cdots, n_b\\
&\bm x \in \mathcal{X}, ~\bm y \in \mathcal{Y}
\end{aligned}
\end{equation}
 
\begin{proposition}
    The relaxed problem \eqref{eq:samp_cv} provides a lower bound for the original robust problem \eqref{eq:GRO}.
\end{proposition}
\begin{proof}
Let $\tilde{\mathcal{F}}_{rob},~\tilde{\mathcal{F}}_{samp},~\tilde{\mathcal{F}}_{nom}$ be the robust, sampled and nominal feasible regions of the relaxed convex problem. Similar to the reformulation of the robust non-convex problem, the auxiliary variables $\bm y$ can be used for the sampled problem and the corresponding feasible region is defined as 
\begin{equation*}
    \mathcal{F}_{samp}:=\left\{\bm x\in \mathcal{X},~ \bm y \in \mathcal{Y} \;\middle|\;
\begin{aligned}
& \tilde{g}_i(\bm{x}, \bm{y}, \bm{u}) \le 0, && \forall\, \bm{u} \in \hat{\mathcal{U}}_i,\; i = 1, \dots, m \\
& y_{bj} = x_b x_j, && b,j = 1, \dots, n_b
\end{aligned}
\right\}
\end{equation*}
and the relaxed feasible region

\begin{equation*}
    \tilde{\mathcal{F}}_{samp}:=\left\{\bm x\in \mathcal{X},~ \bm y \in \mathcal{Y} \;\middle|\;
\begin{aligned}
& \tilde{g}_i(\bm{x}, \bm{y}, \bm{u}) \le 0, && \forall\, \bm{u} \in \hat{\mathcal{U}}_i,\; i = 1, \dots, m \\
& \mu_{bj}(\bm x,\bm y)\leq 0, && b,j = 1, \dots, n_b
\end{aligned}
\right\}
\end{equation*}

The McCormick envelopes provide a relaxation for the original problem, hence $\mathcal{F}_{samp}\subseteq\tilde{\mathcal{F}}_{samp}$. From Proposition \ref{pr:bound} the relaxed problem provides also a lower bound for the robust problem $\mathcal{F}_{rob}\subseteq \mathcal{F}_{samp} \subseteq \tilde{\mathcal{F}}_{samp}$.    
\end{proof}

In the Robust spatial branch-and-bound (RsBB) algorithm we consider a two-level problem approach as in the robust cutting plane algorithm. The upper-level is comprised of the deterministic problems \eqref{eq:samp} and \eqref{eq:samp_cv}. In the lower-level \eqref{eq:maxg} problems evaluate the robustness of the solution of the upper-level non-convex problem. \eqref{eq:samp} is solved to local optimality while  \eqref{eq:samp_cv} and \eqref{eq:maxg} are solved to global optimality. In case of an infeasible The RsBB algorithm is initialized in the root node for the original variable bounds $\mathcal{X}_0=\mathcal{X}$. Note that for both upper-level problems the domain $\mathcal{X}$ and the subset $\hat{\mathcal{U}_i}$ are the same at each node. The domain $\mathcal{X}$ is partitioned into subsets based on the sBB steps of the algorithm and the subset $\hat{\mathcal{U}}_i$ is augmented by any uncertainty samples that violate the robustness of the obtained solution. All waiting nodes are defined for the same sampled uncertainty set $\hat{\mathcal{U}}_i$. In case of an infeasibility detection, the sampled uncertainty set $\hat{\mathcal{U}}_i$ is updated in all waiting nodes and the lower bound solutions need to be re-evaluated. For a given node, $n$ problems $P_{samp}^n$ and $\tilde{P}_{samp}^n$ are defined for $\mathcal{X}_n$, with $\mathcal{F}_n, ~\tilde{\mathcal{F}}_n$  the corresponding feasible regions and $z_n,~\tilde{z}_n$ the objective value of the non-convex and relaxed problems. If $P_{samp}^n$ is found infeasible by the local solver then we set $z_n=+\infty$, the same applies for the relaxed problem. The proposed algorithm for the RsBB implementation is displayed in Algorithm \ref{alg:RSBB} and the Infeasibility Test in Algorithm \ref{alg:InfesTest}.

\begin{lemma}\label{lem:infe_term}
The Infeasibility Test Algorithm \ref{alg:InfesTest} terminates in a finite number of iterations $C$.
\end{lemma}

\begin{proof}
Algorithm \ref{alg:InfesTest} is a robust cutting planes algorithm  applied to nodes for which a solution to $(P_{samp}^n)$ is obtained. Since $(P_{samp}^n)$ is assumed to be tractable, the algorithm remains well-defined at every iteration. Moreover, by Assumptions \ref{as:bounded} - \ref{as:gaff} the solution of the LLP is exact hence the conditions required for finite convergence of Algorithm \ref{alg:RCP} are satisfied. Thus, by the results in \citet{Mutapcic2009} (Section 5.2), Algorithm \ref{alg:InfesTest} terminates in a finite number of iterations $C\leq (\frac{RL}{\delta}+1)^{n_x} $, where $L$ is the Lipschitz constant, $\delta$ the feasibility tolerance of the algorithm and $R$ the radius of the smallest ball entailing $\mathcal{F}_{nom}$.
\end{proof}
The effectiveness of the sBB algorithm heavily depends on the choice of bounding, branching, and node update strategies. Our bounding strategy relies on McCormick relaxations \citep{Mitsos2009,Bompadre2012}, which generate degenerate perfect sequences of under-estimators [\citet{Scott2011}, Theorem 5]. For branching, we adopt the bisection method, which ensures an exhaustive subdivision of the original variable domain [\citet{Horst1996}, Definition IV.10]. The node update strategy follows a best-first approach, selecting the node with the lowest lower bound. This method is node-improving [\citet{Horst1996}, Definition IV.6], meaning it prioritizes promising candidates for early exploration \citep{Locatelli2013}. An additional decision lies in the choice of branching variable. In this work we consider the maximum violation strategy \citep{Smith1999,Tawarmalani2004, Belotti2009} and a pseudoscore strategy \citep{Achterberg2005,Belotti2009}. The use the maximum approximation error strategy is prioritized and if the errors are below the selected tolerance, then the branching variable is selected based on the pseudoscore. We define the branching variable selected via the maximum approximation error as: 
\begin{equation}
    x_{br}=\underset{b,j =1,\cdots,n_{b}}{\argmax}(\max|y_{bj}-x_bx_j|)\}, \quad br\ =1,\cdots,n_{b}
\end{equation}
where $x_b,~x_j ~b,j =1,\cdots,n_{b}$ the variables participating in bilinear formulations and $y_{bj}$ the auxiliary variables of Eqs.\ref{eq:env}.
If the maximum approximation error is below the selected tolerance, the branching variable is selected based on a pseudoscore. For the pseudoscore calculation, strong branching is taking place on the evaluating node $n$ and the variable with the most promising lower bound improvement is selected. In contrast to the approximation error approach, in pseudoscore branching all variables are candidates for branching $br=1,\cdots,n_x$. The intervals for child nodes can be defined as $\delta_i^{n,+}=\overline{x}_i- x_i^b$ and $\delta_i^{n,-}= x_i^b -\underline{x}_i, ~i=1,\cdots, n_x$. Let $\phi_i^{n,+},~\phi_i^{n,-}$ lower bound values of the child nodes for branching in variable $i$ and $\phi^{n}$ the lower bound of the parent node. If any of the child nodes is infeasible, the corresponding lower bound is set to 0. The lower bound improvement of the pseudonodes is evaluated as $\Delta_i^+=\phi_i^{n,+}-\phi^{n}$ and $\Delta_i^+=\phi_i^{n,-}-\phi^{n}$. Let $\varsigma_i^+=\Delta_i^+/\delta_i^+$ be the objective gain for branching on variable $i$. Let $x_{br}$ be the variable selected for branching via pseudoscores:

 \begin{equation}
     x_{br}=\underset{ ~i=1,\cdots,n_{x}}{\argmax}(\varsigma_i^+,\varsigma_i^-)\}, \quad br=1,\cdots,n_{x}
 \end{equation}

\begin{lemma} \label{lem:const}
    The lower bounding operation of the RsBB is consistent.
\end{lemma}
\begin{proof}
    Let $\{M_{k}\}$ be an infinite sequence of nested nodes and $\beta(M_{k})$ the lower bounding operation at node $M_{k}$ defined by McCormick underestimators. Since McCormick relaxations are degenerate perfect relaxations [\citet{Scott2011},Theorem 5], the lower bounding operation is consistent and converges to a feasible solution $\bm x$ of the sampled problem \eqref{eq:samp} $\underset{k\rightarrow \infty}{\lim}\beta(M_{k})\rightarrow f(\bm x^*)$.
\end{proof}

\begin{lemma}\label{lem:conv}
    The lower bounding operation of the RsBB converges to an optimal solution of the sampled problem \eqref{eq:samp}.
\end{lemma}

\begin{proof}
   By the definition of the RsBB algorithm the following conditions holds:
   \begin{itemize}
       \item The bisection partitioning guarantees that the partitioning is exhaustive.
       \item A best-first approach guarantees a bound improving partitioning.
       \item The lower bounding operation is consistent (Lemma \ref{lem:const}).
   \end{itemize}
Hence by [\citet{Horst1996}, Theorem IV.3]$ \underset{k\rightarrow \infty}{\lim}\beta(M_{k})\rightarrow z^*_{samp}$
\end{proof}

\begin{lemma}\label{lem:UBconv}
Assuming that the local solver is convergent, then the upper bounding operation of the RsBB algorithm converges to an optimal solution of the sampled problem \eqref{eq:samp}
\end{lemma}

\begin{proof}
 The local solver is assumed to be convergent under Assumption \ref{as:lips} for regularity conditions on the objective function and constraints. Then for an exhaustive and bound improving partitioning of the compact variable domain the sequence of upper bounding operations is convergent $\underset{k\rightarrow \infty}{\lim}\alpha(M_{k})\rightarrow z^*_{samp}$.
\end{proof}

\begin{theorem}\label{lem:bb_conv}
The branch-and-bound procedure of the RsBB algorithm is finitely convergent to $\epsilon$-optimality.
\end{theorem}
\begin{proof}
    By Lemmas \ref{lem:conv} and \ref{lem:UBconv} $\underset{k\rightarrow \infty}{\lim}\beta(M_{k})= \underset{k\rightarrow \infty}{\lim}\alpha(M_{k})=z^*_{samp}$, hence there exists a node $M_{k^*}$ such that $\alpha_k-\beta_k<\epsilon ,\quad k>k^*$.
\end{proof}

\begin{proposition}\label{lem:rob_opt}
The optimal solution obtained by the RsBB algorithm is also optimal for the robust problem.
\end{proposition}

\begin{proof}
Let $z_k$ denote the objective value of the sampled problem at node $M_{k}$, and let $z^{BF}$  be the best feasible objective value found by the RsBB algorithm up to this node. Let  $\bm x^*_k $ be the corresponding optimal solution for the sampled problem at node $M_{k}$. By construction, RsBB performs a robustness evaluation for all nodes where $ z_k \leq z^{BF} $. This ensures that any solution $\bm x^*_k: z_k \leq z^{BF}$ is also feasible for the robust problem. Hence, the final optimal solution $\bm x^*$ obtained by RsBB is necessarily feasible for the robust problem. By Proposition \ref{pro:opt}, the RsBB solution $\bm x^*$ is also optimal for the robust problem.
\end{proof}

\begin{theorem}
    The RsBB algorithm terminates in a finite number of iterations at a robust optimal solution.
\end{theorem}

\begin{proof}
    The Infeasibility Test algorithm has a finite termination by Lemma \ref{lem:infe_term} and the branch-and-bound procedure of the RsBB algorithm is finitely convergent by Theorem \ref{lem:bb_conv}, hence the RsBB algorithm terminates in a finite number of iterations. By Proposition \ref{lem:rob_opt} the obtained solution is optimal for the robust problem.
\end{proof}

\begin{algorithm}
\caption{Infeasibility Test Algorithm}
\label{alg:InfesTest}
\begin{algorithmic}[1] 
    \Require Local solution of $(P_{samp}^n)$ at node $n$, i.e., $\bm x^*$
    \Require Feasibility tolerance $\delta>0$
    \Require Subset of sampled uncertain parameters $\hat{\mathcal{U}_i}$
    \Ensure Updated set of sampled uncertain parameters $\hat{\mathcal{U}}_i'$
    \For{$i=1,\cdots,m$}
    \State Solve \eqref{eq:maxg} $\rightarrow u^*$
    \If {$g_i(\bm x^*,\bm u^*) > \delta$}
        \State Add new sample to $\hat{\mathcal{U}}_i \leftarrow \hat{\mathcal{U}}_i \cup \{\bm u^*\}$
    \EndIf
    \EndFor
\If {$\underset{i=1,\cdots,m}{\max} g_i(\bm u^*, \bm x^*) > \delta$}
       \State Set $\hat{\mathcal{U}}_i'= \hat{\mathcal{U}}_i$
        \State Solve $(P_{samp}^n) \rightarrow \bm x^*$
        \State Go to Step 1
    \EndIf
    \State \textbf{Terminate:} Robust node, continue to BB
\end{algorithmic}
\end{algorithm}

\begin{algorithm}
\caption{The RsBB Algorithm}
\label{alg:RSBB}
\begin{algorithmic}[1] 
\Require Numeric tolerance $tol>0$ and fathoming tolerance $\epsilon>0$
    \Require Set of waiting nodes $\mathcal{W} := \{0\}$, set of current nodes $\mathcal{C} := \emptyset$
    \Require Set of uncertainty samples $\hat{\mathcal{U}}= \underset{i=1,\cdots,m}{\bigcup} \hat{\mathcal{U}_i}=\underset{i=1,\cdots,m}{\bigcup}  \{ u_{i,0}\}
   $
    \State Initialize variable bounds $\underline{\bm x},\overline{\bm x}$
    \State Initialize best possible solution $z^{BP} \gets \infty$ and best found solution $z^{BF} \gets \infty$

    \While{$\mathcal{W} \neq \emptyset$}
        \State $z^{BP} \gets \underset{n\in \mathcal{W}}{\min}{\tilde{z}_n}$
        \State Reset $\mathcal{C} \gets \emptyset$
        \For{$n \in \mathcal{W}$}
            \If{$\tilde{z}_n - z^{BP}\leq tol$}
                \State Add current node to $\mathcal{C} \gets \mathcal{C} \cup \{n\}$
            \EndIf
        \EndFor
        
        \For{$n \in \mathcal{C}$}
            \State Remove $n$ from waiting nodes $\mathcal{W} \gets \mathcal{W} \setminus \{n\}$
            \State Solve $P_{samp}^n \rightarrow z_n$
            
            \If{$z_n \leq z^{BF}$}
                \State Perform Infeasibility Test (Algorithm \ref{alg:InfesTest}) $\rightarrow \hat{\mathcal{U}}'$
                \If{New violations detected: $|\hat{\mathcal{U}}'| > |\hat{\mathcal{U}}|$}
                    \State Update $\hat{\mathcal{U}} \gets \hat{\mathcal{U}}'$
                    \State Solve $\tilde{P}_{samp}^n,~ \forall n \in \mathcal{W} \rightarrow \tilde{z}_n $
                \EndIf
                \State Set $z^{BF} \gets \min\{z_n, z^{BF}\}$
                \State Solve $\tilde{P}^n_{samp} \rightarrow \tilde{z}_n$
                \State Select branching variable
                \State Generate child nodes $n', n''$
                \State Append $\mathcal{W} \gets \mathcal{W} \cup \{n',n''\}$
                
                \For{$w \in \mathcal{W}$}
                    \If{$\tilde{z}_w(1 - \epsilon) \geq z^{BF}$}
                        \State Fathom node $\mathcal{W} \gets \mathcal{W} \setminus \{w\}$
                    \EndIf
                \EndFor
            \EndIf
        \EndFor
    \EndWhile

    \State \textbf{Terminate:} Robust optimal solution obtained at $z^{BF}$
\end{algorithmic}
\end{algorithm}

\subsection{Illustrative example}
The implementation of the RsBB algorithm is demonstrated in a non-convex QCQP toy problem $(P_{nom,toy})$ defined by Eqs \eqref{eq:toy-start}-\eqref{eq:toy-end}. We consider box uncertainty for $u\in \mathcal{U}=[2,6]$ around the nominal value $u_0=4$. The corresponding relaxed problem using McCormick envelopes will be denoted as $(\tilde{P}_{toy})$. The feasible regions for the toy nominal, sampled and robust problems are depicted in Fig. \ref{fig:feas_reg}. The robust feasible region corresponds to a set of equations when no more feasibility violations are detected, while the sampled region corresponds to intermediate robust cutting plane iterations.
\begin{align}
        \min-2x_1x_2\label{eq:toy-start}\\
        \text{s.t.}~ u_0 x_1x_2+2x_1+2x_2\leq 3\label{eq:toy_unc}\\
        0.1 -(x_1-0.5)^2 -(x_2-0.5)^2\leq 0\\
        x_2-0.09x_2-0.5 \leq 0\\
        \textbf{x}\in \mathcal{X}=[0,1]^2\label{eq:toy-end}
\end{align}
The contour plots for the objective value in the feasible regions are displayed in Figure \ref{fig:toy_contour} which also entails the obtained solutions of the examined nodes in rhombuses. For the root node and node 1 a trivial solution was found by the local solver hence no robust cuts were added to the problem. Node 2 obtained an optimal solution for the nominal problem resulting in an infeasibility cut. The robust optimal solution was obtained on the same node for a tighter sampled feasible region. The descendand nodes are not depicted on the figure but can be found in  the detailed RsBB tree representation of Figure \ref{fig:toy_tree}. The implementation of the RsBB algorithm is the following:\\
\textbf{Step 0:}
Initialize the sampled uncertainty set $\hat{\mathcal{U}}=\{4\}$ with the nominal uncertainty value. Solve $(P_{nom,toy})$ with local solver. The obtained solution at root node is $z_0=0$ which is set as the best found $z^{BF}=z^*=0$ [Step 2, Algorithm \ref{alg:RSBB}]  The solution trivially satisfies the uncertain constraint in Eq.\eqref{eq:toy_unc} hence no cuts are added [Steps 14, 15, Algorithm \ref{alg:RSBB}].\\
\textbf{Step 1:}
Solve the relaxed problem $(\tilde{P}_{samp,toy})$ and generate two new (1,2) child nodes based on maximum approximation error [Steps 21-24, Algorithm \ref{alg:RSBB}].\\
\textbf{Step 2:}
Select new node as $n=1$ based on lowest-lower bound $\tilde{z}_1<\tilde{z}_2$ [Step 4, Algorithm \ref{alg:RSBB}].\\
\textbf{Step 3:}
Solve $(P_{samp,toy})$ with local solver. The obtained optimal value is $z_1=0$ [Steps 12, 13, Algorithm \ref{alg:RSBB}]. The Infeasibility Test is trivial for this node as well [Steps 14, 15, Algorithm \ref{alg:RSBB}].\\
\textbf{Step 4:}
Solve the relaxed problem $(\tilde{P}_{samp,toy})$ and generate two new (3, 4) child nodes based on maximum approximation error[Steps 21-24, Algorithm \ref{alg:RSBB}]. None of the waiting nodes is fathomed.\\
\textbf{Step 5:}
Select new node as $n=2$ based on lowest-lower bound $\tilde{z}_2<\tilde{z}_3<\tilde{z}_4$ [Steps 4-10, Algorithm \ref{alg:RSBB}].\\
\textbf{Step 6:}
Solve $(P_{samp,toy})$ with local solver [Steps 12, 13, Algorithm \ref{alg:RSBB}]. The obtained optimal value is $z_2=-0.45$.\\
\textbf{Step 7:}
Perform the Infeasibility Test [Steps 14, 15, Algorithm \ref{alg:RSBB}]. The sampled point $u^*=6$ was found to violate Eq. \eqref{eq:toy_unc} and was appended in $\hat{\mathcal{U}}\cup \{6\}$ [Steps 1-8, Algorithm \ref{alg:InfesTest}].\\
\textbf{Step 8:}
Solve $(P_{samp,toy})$ with local solver and the new feasible region [Steps 9, Algorithm \ref{alg:InfesTest}]. The obtained optimal value is $z_2=-0.36$ [Step 12 Algorithm \ref{alg:InfesTest}].\\
\textbf{Step 9:} 
Update best found solution as $z^*_2=-0.36$ [Step 20, Algorithm \ref{alg:RSBB}].\\
\textbf{Step 10:}
Solve the relaxed problem $(\tilde{P}_{samp,toy})$ and generate two new (5,6) child nodes based on maximum approximation error [Steps 21-24, Algorithm \ref{alg:RSBB}].\\
\textbf{Step 11:}
Fathom nodes for which $\tilde{z}_n(1-\epsilon)>z^*$. Nodes 4 and 6 are fathomed [Steps 25-29, Algorithm \ref{alg:RSBB}].\\
\textbf{Step 12:}
Select new node based on lowest-lower bound [Steps 4-10, Algorithm \ref{alg:RSBB}].\\
Steps 8 to 12 are repeated for the remaining nodes, the corresponding solutions for each node can be found in Figure \ref{fig:toy_tree}. The RsBB terminates in node 14 and the robust optimal solution $z^*=-0.36$.

\begin{figure}[H]
    \centering
    \includegraphics[width=0.4\linewidth]{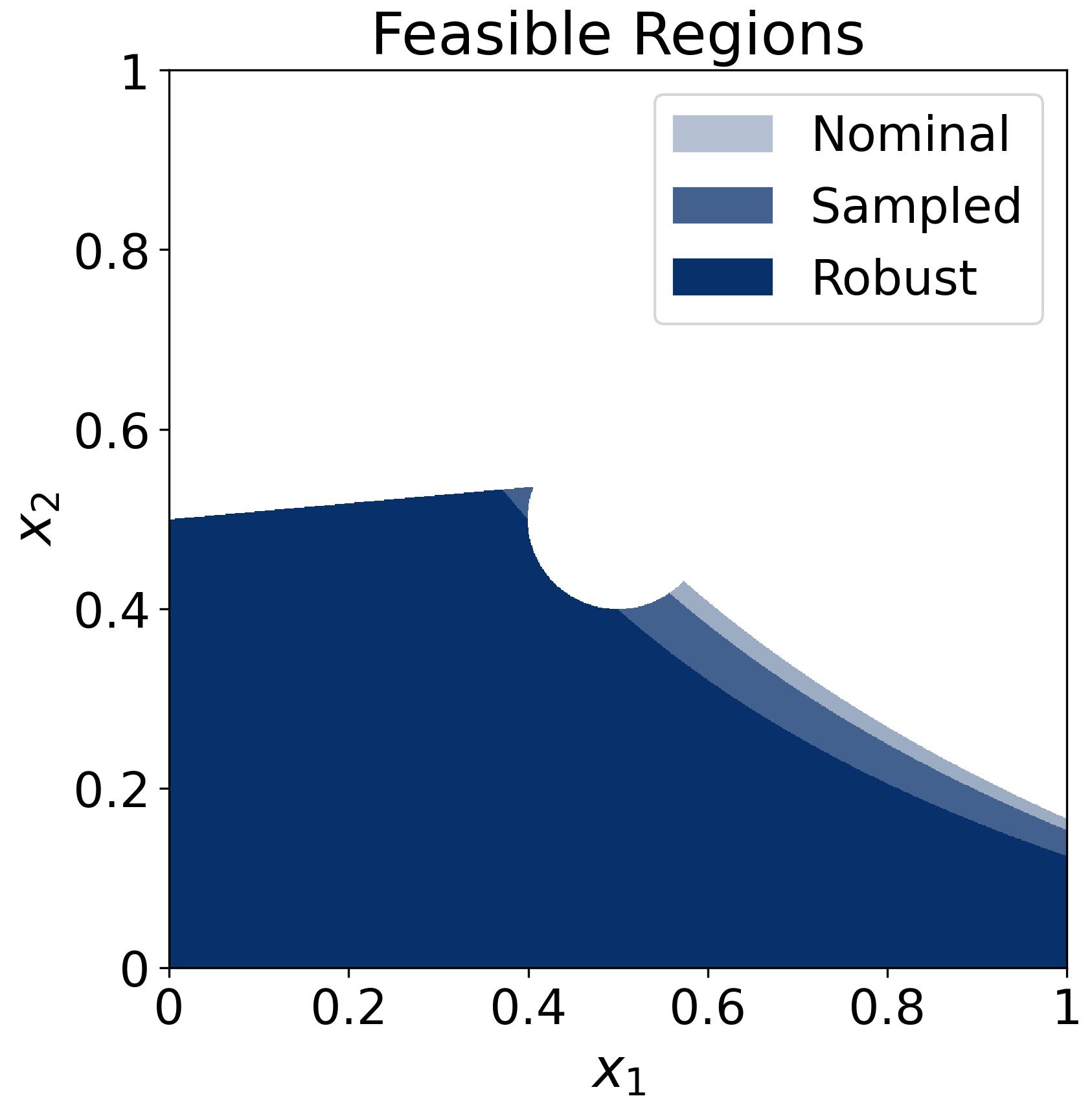}
    \caption{Feasible regions for nominal sampled and robust toy problems}
    \label{fig:feas_reg}
\end{figure}

\begin{figure}[H]
    \centering
    \includegraphics[width=\linewidth]{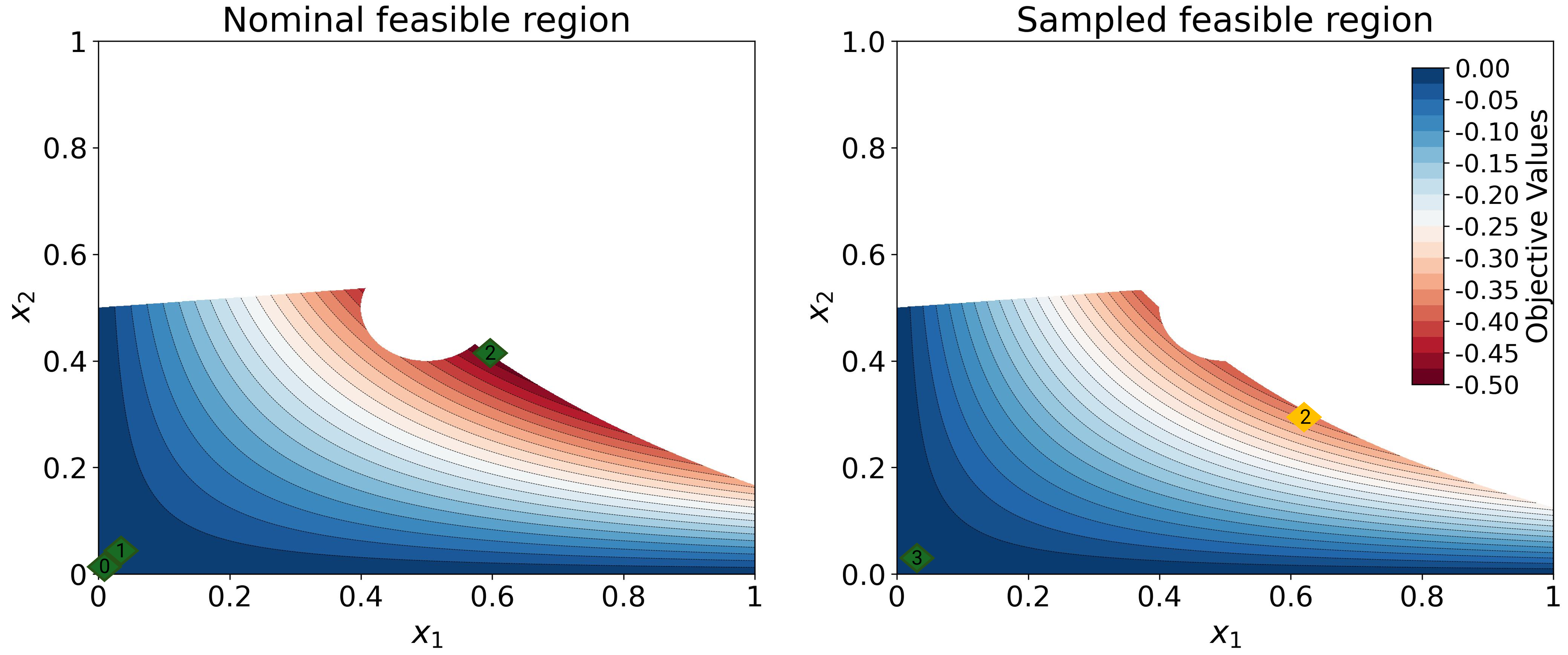}
    \caption{Contour plot for toy problem over the nominal and sampled feasible regions. Rhombuses correspond to objective values found at denoted nodes.}
    \label{fig:toy_contour}
\end{figure}

\begin{figure}
\resizebox{\columnwidth}{!}{%
\begin{tikzpicture}[
    every node/.style={circle, draw=black,minimum size=7mm},
    level 1/.style={sibling distance=60mm},
    level 2/.style={sibling distance=40mm},
    level 3/.style={sibling distance=30mm},crossed/.style={  append after command={
    \pgfextra{\draw[-]
      (\tikzlastnode.south west) -- (\tikzlastnode.north east)
      (\tikzlastnode.north west) -- (\tikzlastnode.south east);}
  }
}
]
\node[label=left:{\begin{tabular}{l}   $z_0^*=0$ \\    $\tilde{z}_0=-0.5$ \end{tabular}},label=right:{$\hat{\mathcal{U}}=\{4\}$}] {0}
  child[->] { node[label=left:{\begin{tabular}{l}    $z_1^*=0$\\    $\tilde{z}_1=-0.5$ \end{tabular}}] {1}
    child[->] { node[very thick,label=left:{\begin{tabular}{l}    $z_3=0$ \\    $\tilde{z}_3=-0.39$ \end{tabular}}] {3}
      child[->] { node[dashed,very thick,label=left:{$\tilde{z}_{11}=-0.22$}] {11} }
      child[->] { node[dashed,very thick,label=left:{$\tilde{z}_{12}=-0.33$}] {12} }
    }
    child[->] { node[dashed,very thick,label=left:{$\tilde{z}_4=-0.35$}] {4} }
  }
  child[->] { node[draw=teal,very thick] (n2) {2}
    child [->]{ node[very thick,label=right:{\begin{tabular}{l}    $z_5^*=-0.36$ \\    $\tilde{z}_5=-0.40$ \end{tabular}}] {5}
    child [->]{ node[dashed,very thick,label=left:{$\tilde{z}_7=-0.35$}] {7} }
    child [->]{ node[very thick,label=right:{\begin{tabular}{l}    $z_8^*=-0.36$ \\    $\tilde{z}_8=-0.40$\end{tabular}}] {8} 
    child [->]{ node [very thick,label=left:{\begin{tabular}{l}    $z_9^*=-0.36$ \\    $\tilde{z}_9=-0.38$ \end{tabular}}] {9}
    child [->]{ node[dashed,very thick,label=left:{$\tilde{z}_{13}=:$ infeas}] {13} }
    child [->]{ node [very thick,label=right:{\begin{tabular}{l}    $z_{14}^*=-0.36$ \\    $\tilde{z}_{14}=-0.36$ \end{tabular}}] {14} }}
    child [->]{ node[dashed,very thick,label=right:{$\tilde{z}_{10}=:$ infeas}] {10} }}}
    child [->]{ node[dashed,very thick,label=right:{$\tilde{z}_6=:$ infeas}] {6} }
  };
\node[above=-15mm of n2,xshift=15mm, draw=none, fill=none, inner sep=0pt] (anno1){$z_2=-0.45$, feasibility violation};
\node[right=5mm of n2, draw=none, fill=none, inner sep=0pt] (anno2) { \begin{tabular}{l}$z_2^*=-0.36$\\$\tilde{z}_2=-0.38$\end{tabular}};
\draw[->] (anno1.east) to [out=0,in=0]  node[midway,draw=none, fill=none, right] {$\hat{\mathcal{U}}\leftarrow \hat{\mathcal{U}} \cup \{6\}$}  (anno2.east);
\end{tikzpicture}
}
    \caption{RsBB tree at termination for the toy problem. We denote as $z_n$ and $\tilde{z}_n$ the solutions of $P_{samp,toy}$ and $\tilde{P}_{samp,toy}$ at node $n$. We note as $infeas$ the nodes for which $\tilde{P}_{samp,toy}$ was found infeasible. Nodes with dashed lines correspond to fathomed nodes. Nodes with thin lines ($n=0,1$) are evaluated for the original sampled uncertainty set $\hat{\mathcal{U}}=\{4\}$. At node 2 a feasibility violation is detected and the sampled uncertainty set is updated accordingly to $\hat{\mathcal{U}}=\{4,6\}$. The remaining nodes with thick lines ($n=3,\cdots,14$) are evaluated for the updated sampled uncertainty set.}
    \label{fig:toy_tree}
\end{figure}

\section{Problem description}\label{prob}
For the purpose of this study, we evaluate the performance of RsBB for the pooling problems. We use the pq-formulation of the pooling problem \citep{Sahinidis2005}, resulting in a quadratically constrained quadratic program (QCQP). Bilinear terms are manifested both in the objective function and the quality constraints.
\subsection{The pooling problem}
In the standard pooling formulation, given an existing infrastructure, the optimal flows  between input streams, mixing pools and output product streams need to be determined. The objective is to minimize the cost of the process under storage capacity, demand satisfaction and product quality constraints. \citet{Bental94} proposed the q-formulation of the pooling problem in order to derive smaller in size dual problems. Later on, \citet{Sahinidis2005} proposed an alternative pq-formulation via the use of two RLT constraints which improved the performance of the state-of-the-art solvers at that time, while not changing the problem structure. We introduce the auxiliary variables $v_{ilj}$ to replace the bilinear terms. Let $\Pi_{nom}$ be the original non-convex pooling problem of cost minimization:
\begin{align}
{\text{min}}~\sum_{ilj}c_iv_{ilj} -\sum_{j}d_j\sum_{l}y_{lj}+\sum_{j}d_j\sum_{i}z_{ij}\label{eq:obj}\\
\text{s.t.}~ \sum_{lj}v_{ilj}+ \sum_j z_{ij} \leq A_i, \quad \forall i \\
\sum_j y_{lj} \leq S_l, \quad \forall l\\
\sum_l y_{lj}+\sum_i z_{ij}\leq D_j, \quad \forall j\\
\sum_i v_{ilj}=y_{lj}, \quad \forall l\\
\sum_l y_{lj}+ \sum_i z_{ij}=f_j, \quad \forall j\label{eq:flow}\\
P_{jk}^L f_j\leq \sum_{il}C_{ik} v_{ilj}+ \sum_i C_{ik} z_{ij}\leq P_{jk}^U f_j, \quad \forall j,k \label{eq:unc_quality}\\
v_{ilj}=q_{il}y_{l,j}~\forall i,l,j\label{eq:aux}
\end{align}  
To derive the convex relaxation $\tilde{\Pi}_{nom}$ of the pooling problem $\Pi_{nom}$, Eq. \eqref{eq:aux} is replaced by the equivalent McCormick envelopes Eqs.\eqref{eq:mc1}-\eqref{eq:mc4} \cite{McCormick1976}.
\begin{align}
v_{ilj} \geq \underline{q}_{il}y_{lj}+q_{il}\underline{y}_{lj} -\underline{q}_{il} \underline{y}_{lj},\quad \forall i,l,j\label{eq:mc1}\\
v_{ilj} \leq \overline{q}_{il}y_{lj}+q_{il}\underline{y}_{lj} -\overline{q}_{il} \underline{y}_{lj},\quad \forall i,l,j \\
v_{ilj} \geq \overline{q}_{il}y_{lj}+q_{il}\overline{y}_{lj} -\overline{q}_{il} \overline{y}_{lj},\quad \forall i,l,j \\
v_{ilj} \leq \underline{q}_{il}y_{lj}+q_{il}\overline{y}_{lj} -\underline{q}_{il} \overline{y}_{lj},\quad \forall i,l,j\label{eq:mc4}
\end{align}  
In the pooling problem, uncertainty can be manifested in any of the model parameters, such as the feed component cost, product price and feed component quality. For this study we consider uncertainty in the feed component quality, i.e. $C_{ik}$ in Eq. \eqref{eq:unc_quality}. Following the notation in \cite{Li2011}, the true value of the uncertain level quality  can be formulated as follows:
\begin{equation}\label{eq:unc-def}
\tilde{C}_{ik}=C_{ik}+\xi_{ik} \hat{C}_{ik}    
\end{equation}
where  $C_{ik}$ represents the nominal value, $\hat{C}_{ik}$ the constant perturbation (which is positive) and $\xi_{ik}$ is a vector of random variables which are subject to uncertainty. Given a selected uncertainty set $\mathcal{U}_{p}$, the  deterministic quality constraints in Eq. \eqref{eq:unc_quality} are replaced by their robust formulations in Eqs.\eqref{eq:rob-qual-up} and\eqref{eq:rob-qual-lo} resulting in the SIP problems  $\Pi_{rob},~\tilde{\Pi}_{rob}$ \citep{Bental02RO}.
\begin{align}
 \sum_{il}\tilde{C}_{ik} v_{ilj}+ \sum_i \tilde{C}_{ik} z_{ij}\leq P_{jk}^U f_j, \quad \forall \boldsymbol{\xi}\in \mathcal{U}_{p},~ \forall j,k \label{eq:rob-qual-up}\\ 
 P_{jk}^L f_j\leq \sum_{il}\tilde{C}_{ik} v_{ilj}+ \sum_i \tilde{C}_{ik} z_{ij}, \quad \forall \boldsymbol{\xi}\in \mathcal{U}_{p},~ \forall j,k \label{eq:rob-qual-lo}
\end{align}
\subsection{Uncertainty set modelling}
In this study we consider the case where $\mathcal{U}_{p}$ is a convex $p$-normed uncertainty set with $p\in \{\infty,2,1\}$, i.e. box, ellipsoidal and polyhedral uncertainty. \citet{Soyster73} was the first to study linear problems under box uncertainty using the dual reformulation method, resulting in the most conservative solution, i.e. worst-case solution.
Later on, \citet{Bental00} aiming at a less conservative approach, proposed the ellipsoidal uncertainty set defined by the $2$-norm. Despite the fact that the ellipsoidal set reduces the conservatism of the problem, it introduces extra nonlinear terms. To this end, \citet{Bertsimas04} proposed an alternative formulation assuming that it it unlikely that all uncertain parameters would be affected by the maximum perturbation. The corresponding polyhedral set is defined by the $1$-norm. Given that Eqs. \eqref{eq:rob-qual-up} and \eqref{eq:rob-qual-lo} are linear in $\xi_{ik}$ both dual reformulation and robust cutting planes can be used to address the robust problem $\Pi_{rob}$. The deterministic dual reformulations of Eqs. \eqref{eq:rob-qual-up} and\eqref{eq:rob-qual-lo} derived in Eqs. \eqref{eq:dual-up} and\eqref{eq:dual-lo} using the $\rho_{jk}$ according to the selected uncertainty set as displayed in Table \ref{tab:dual}.

\begin{table}[h!]
\centering
\caption{Uncertainty set definitions and dual reformulations for the examined uncertainty set types.}
\label{tab:dual}
\begin{tabular}{c c c}
\hline
\textbf{Uncertainty set}  & $\mathcal{U}$ & $\bm{\rho}_{jk}$  \\ \hline 
\noalign{\vskip 0.5em} 
Box & $\{\boldsymbol{\xi}| ~|\xi_{i}|\leq \Psi\,~\forall i\}$&$\Psi\big(\sum_{il}\hat{C}_{ik} q_{il}y_{lj}+ \sum_i \hat{C}_{ik} z_{ij} \big)$  \\ 
\noalign{\vskip 0.5em} 
Ellipsoidal & $\{\boldsymbol{\xi}| ~ \sqrt{\sum_{i}\xi_{i}^2} \leq \Omega \}$& $\Omega\bigg(\sqrt{\sum_{i}\hat{C}^2_{ik}\big( \sum_{l}q_{il}y_{lj}+ z_{ij}\big)^2}\bigg )$  \\ 
\noalign{\vskip 0.5em} 
Polyhedral & $\{\boldsymbol{\xi}| ~ \sum_{i}|\xi_{i}| \leq \Gamma_k \}$&$\Gamma_k\underset{i}{\max}\hat{C}_{ik}(\sum_l q_{il}y_{lj}+z_{ij})$ \\ \hline
\end{tabular}
\end{table}

    \begin{align}
\sum_{il}C_{ik} q_{il}y_{lj}+ \sum_i C_{ik} z_{ij}  - P_{jk}^U v_j + \rho_{jk} \leq 0 \label{eq:dual-up} \\ 
P_{jk}^L v_j- \sum_{il}C_{ik} q_{il}y_{lj}- \sum_i C_{ik} z_{ij} +  \rho_{jk} \leq 0 \label{eq:dual-lo} 
\end{align}

For the robust cutting planes algorithm, consider the sampled quality constraints Eqs. \eqref{eq:samp-qual-up}, \eqref{eq:samp-qual-lo}. The robustness of each $\Pi_{samp}$ solution is evaluated by the infeasibility test problem $(T^{pool}_{jk})$, $\forall j,k $ and both upper and lower quality constraints. For each sampled $\bm{\xi}$ that violates a quality constraint, the sampled uncertainty set $\hat{\mathcal{U}}_p$ is augmented by the corresponding $\bm{\xi}$ value. 

\begin{align}
 \sum_{il}\tilde{C}_{ik} v_{ilj}+ \sum_i \tilde{C}_{ik} z_{ij}\leq P_{jk}^U f_j, \quad \boldsymbol{\xi}\in \hat{\mathcal{U}}_{p},~ \forall j,k \label{eq:samp-qual-up}\\ 
 P_{jk}^L f_j\leq \sum_{il}\tilde{C}_{ik} v_{ilj}+ \sum_i \tilde{C}_{ik} z_{ij}, \quad \boldsymbol{\xi}\in \hat{\mathcal{U}}_{p},~ \forall j,k \label{eq:samp-qual-lo}
\end{align}
\begin{equation}\tag{$T^{pool}_{jk}$}\label{eq:maxg_pool}
    \begin{aligned}
    \begin{rcases}
\max~\sum_{il}\tilde{C}_{ik} v^*_{ilj}+ \sum_i \tilde{C}_{ik} z^*_{ij}- P_{jk}^U f^*_j \\
    \text{s.t.}\quad\tilde{C}_{ik}=C_{ik}+\xi_{ik} \hat{C}_{ik}\\
    ~\boldsymbol{\xi}\in \mathcal{U}_p
    \end{rcases}{\forall j,k}
    \end{aligned}
\end{equation}

 Table \ref{tab:models} summarizes the notation for the different problems and the equations comprising them providing a correspondence to the general problem notation used for the RsBB methodology.

 \begin{table}[h!]
\centering
\caption{Overview of the examined pooling problem  categorized by optimization level. Each problem is defined by its notation and the corresponding set of equations that describe its feasible region.}\label{tab:models}
\begin{tabular}{|c|c|c|c|}
\hline
\textbf{Optimization level} & \textbf{Problem} & \textbf{Notation} & \textbf{Equations} \\ \hline \hline
\multirow{6}{*}{Outer-level} 
& Nominal pooling & $\Pi_{nom}$ & \eqref{eq:obj}-\eqref{eq:aux} \\ 
& Robust pooling & $\Pi_{rob}$ & \eqref{eq:obj}-\eqref{eq:flow}, \eqref{eq:aux}, \eqref{eq:rob-qual-up}, \eqref{eq:rob-qual-lo} \\ 
& Sampled pooling & $\Pi_{samp}$ & \eqref{eq:obj}-\eqref{eq:flow}, \eqref{eq:aux}, \eqref{eq:samp-qual-up}, \eqref{eq:samp-qual-lo}\\ 
& Relaxed nominal pooling& $\tilde{\Pi}_{nom}$ & \eqref{eq:obj}-\eqref{eq:unc_quality}, \eqref{eq:mc1}-\eqref{eq:mc4} \\ 
& Relaxed robust pooling& $\tilde{\Pi}_{rob}$ & \eqref{eq:obj}-\eqref{eq:flow}, \eqref{eq:mc1}-\eqref{eq:mc4}, \eqref{eq:rob-qual-up}, \eqref{eq:rob-qual-lo} \\
& Relaxed sampled pooling& $\tilde{\Pi}_{samp}$ & \eqref{eq:obj}-\eqref{eq:flow}, \eqref{eq:mc1}-\eqref{eq:mc4}, \eqref{eq:samp-qual-up}, \eqref{eq:samp-qual-lo}\\ \hline
\multirow{1}{*}{Inner-level} 
& Quality robustness & $T_{pool}$ & \eqref{eq:maxg_pool} \\ \hline
\end{tabular}
\end{table}

\section{Computational experiments}\label{results}
\subsection{Case studies}
The performance of the RsBB algorithm is evaluated for 10 benchmark pooling instances, details of the problems can be found in Table \ref{tab:problem_stats}. For the uncertain inlet quality, we set $\hat{C}_{ik}=C_{ik}$. The uncertain parameters are defined by box, ellipsoidal and polyhedral uncertainty sets for six different uncertainty set sizes, i.e. $\Psi, ~\Omega,~ \Gamma : =\{0.05,0.1,0.15,0.2,0.25,0.3\}$ as introduced in Table \ref{tab:dual}. The selected modelling environment is Pyomo v6.9.3 \citep{Pyomo} running on Python 3.9.23. For the RsBB algorithm HiGHS v.1.11.0 \citep{highs} is used as the LP solver and conopt v4.35 \citep{Drud1994} as the local NLP solver. The proposed algorithm is compared with PyROS v1.2.9 using Gurobi v12.0.3 \citep{gurobi} as the global solver. The default settings were used for all solvers. The deterministic robust counterparts are evaluated as well using Gurobi v12.0.3  as global solver. We set a time limit of 1 hour for all of the examined methods. The computational experiments were carried on Windows 10 operating system running an Intel Core i9-10900K CPU @ 3.70GHZ processor with 10 cores and 20 logical processors and 32.0 GB RAM. The examined instances were processed serially. 

\begin{table}[htpb]
    \centering
    \caption{Model statistics for examined pooling problems}
    \begin{tabular}{lcccccccc}
        \toprule
        Problem & Feeds & Pools & Products & Qualities & Var. & Eq. QP & Eq. LP & Ref. \\
        \midrule
        haverly1 & 3  & 1  & 2  & 1  & 10  & 17  & 29 &\citet{Haverly1978} \\
        haverly2 & 3  & 1  & 2  & 1  & 13  & 20  & 38 &\citet{Haverly1978} \\
        haverly3 & 3  & 1  & 2  & 1  & 10  & 17  & 29 &\citet{Haverly1978} \\
        bental4  & 4  & 1  & 2  & 1  & 13  & 21  & 39 &\citet{Bental94} \\
        bental5  & 13 & 3  & 5  & 2  & 92  & 121 & 301 &\citet{Bental94} \\
        foulds2  & 6  & 2  & 4  & 1  & 36  & 46  & 94 &\citet{Foulds1992} \\
        adhya1   & 5  & 2  & 4  & 4  & 33  & 62  & 122 &\citet{Adhya1999} \\
        adhya2   & 5  & 2  & 4  & 6  & 33  & 70  & 130 &\citet{Adhya1999} \\
        adhya3   & 8  & 3  & 4  & 6  & 52  & 94  & 190 &\citet{Adhya1999} \\
        adhya4   & 8  & 2  & 5  & 4  & 58  & 95  & 215 &\citet{Adhya1999} \\
        \bottomrule
    \end{tabular}
    \label{tab:problem_stats}
\end{table}

\subsection{Results}
A comparison of the RsBB algorithm with state-of-the-art methods in terms of \% of problems solved within 1 hour of run time can be found in Figure \ref{fig:inst_solved}. Dual-Gurobi corresponds to solving the dual problem, using the dual constraints for the different uncertainty types, with global solver Gurobi. PyROS-Gurobi corresponds to using PyROS solver with Gurobi as a global solver both for the upper and lower-level problems. The RsBB algorithm is using conopt4 for the upper and lower-level problems for box and polyhedral sets and Gurobi for the lower-level problems under ellipsoidal uncertainty. For each of the 10 benchmark problems of Table \ref{tab:problem_stats} 6 different uncertainty sizes are evaluated, resulting in 60 instances in total for each uncertainty type. All three approaches solve the examined instances to robust optimality for box uncertainty set. Dual reformulation and RsBB solve all instances with polyhedral uncertainty, whereas PyROS attains the robust optimal solution for 40\% of the cases. The instances modelled with an ellipsoidal uncertainty set pose a challenge to the examined approaches, with Dual reformulation and PyROS solving 90\% of the instances and RsBB 92\% within the imposed time limit. Details for the instances that were not solved to robust optimality are displayed in Table \ref{tab:problem_stats}. The bental5 and adhya1-4 instances result in a time limit exceeded termination, these instances correspond to the problems with higher dimensionality and complexity. For the ellipsoidal uncertainty set for $\Omega<0.3$ an optimal solution with relative gap less than 3\% is attained by the Dual reformulation and less that 2\% by RsBB, which correspond to robust feasible solutions. Imposing $\Omega=0.3$ significantly impeded the computational complexity of bental5 resulting in a feasible solution with 15\% gap even after 24h of run time with Gurobi. For the box and ellipsoidal set the Dual reformulation terminates in less than 1 sec for the instances solved to robust optimality . For the ellipsoidal set 85\% of the instances are solved in less than 1 sec. The RsBB and PyROS methods are performing similarly for the instances solved in  less than 1 sec, however as the computational demand increases the RsBB results in reduced CPU time. While \citep{Wiebe2019} found similar performance in terms of median CPU time between the box and polyhedral set and a higher median CPU for the ellipsoidal set, only 83\% of the instances were solved for the polyhedral set. In this study, the CPU time for the Dual reformulation is higher for the polyhedral set, yet all instances could be solved within 1 hour. This is to highlight that the performance of the global solver dictates the convergence to the robust optimal solution. The results of Figure \ref{fig:inst_solved} suggest that integrating the robust cutting planes in a sBB framework can improve the performance of solely using robust cutting planes with a global solver. The convergence rate of the robust optimization methodologies under ellipsoidal uncertainty set is impeded by the problem complexity. In pooling problems with a high number of input feeds and inlet quality parameters the use of the ellipsoidal could result in significant computational delays, especially following the dual reformulation approach. 

\begin{figure}[H]
    \centering
    \includegraphics{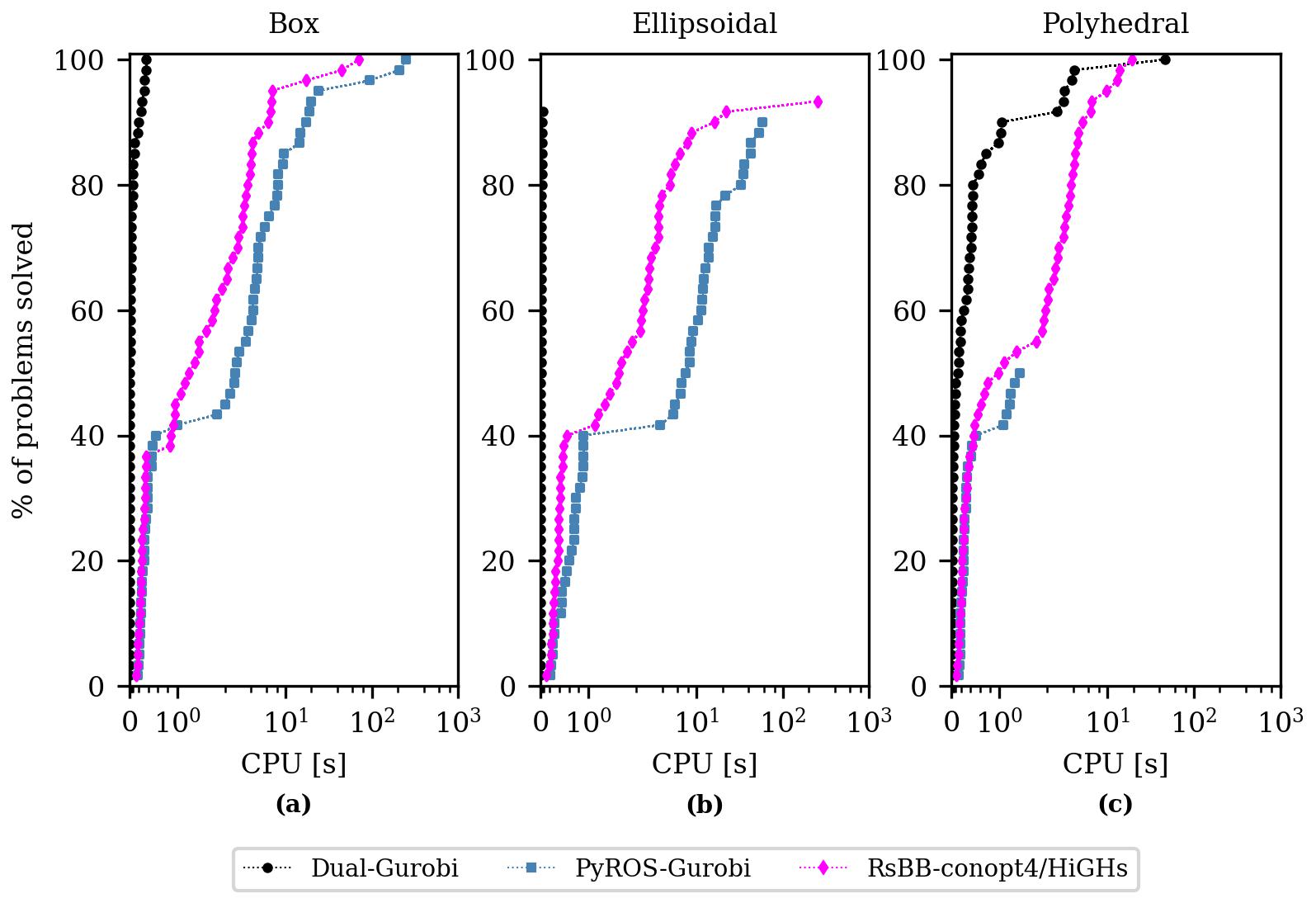}
    \caption{Cumulative CPU time for solved instances via Dual reformulation, PyROS and RsBB methods for different uncertainty types}
    \label{fig:inst_solved}
\end{figure}

\begin{sidewaystable}[]
\centering
\caption{Solution statistics for the instances not solved in Figure \ref{fig:inst_solved}}
\begin{tabular}{llccccccc}
\toprule
\textbf{Problem} & \textbf{Set type} & \textbf{Set size} &
\multicolumn{2}{c}{\textbf{Dual reformulation}} &
\multicolumn{1}{c}{\textbf{PyROS}} &
\multicolumn{2}{c}{\textbf{RsBB}} \\
\cmidrule(lr){4-5} \cmidrule(lr){6-6} \cmidrule(lr){7-8}
 &  &  & \textbf{Gap} & \textbf{Termination} & \textbf{Termination} & \textbf{Gap} & \textbf{Termination} \\
\midrule
bental5 & Ellipsoidal & 0.05 & $<$0.01\% & Optimal solution &  Time Limit Exceeded & $<$0.01\% & Optimal solution \\
bental5 & Ellipsoidal & 0.10 & 0.4\% & Time Limit Exceeded  & Time Limit Exceeded & $<$0.01\% & Optimal solution \\
bental5 & Ellipsoidal & 0.15 & 1.3\% & Time Limit Exceeded  & Time Limit Exceeded & 1.00\% & Time Limit Exceeded \\
bental5 & Ellipsoidal & 0.20 & 1.7\% & Time Limit Exceeded & Time Limit Exceeded & 1.30\% & Time Limit Exceeded \\
bental5 & Ellipsoidal & 0.25 & 2.8\% & Time Limit Exceeded  & Time Limit Exceeded & 1.80\% & Time Limit Exceeded \\
bental5 & Ellipsoidal & 0.30 & 48.0\% & Time Limit Exceeded  & Time Limit Exceeded & 46.70\% & Time Limit Exceeded \\
bental5 & Polyhedral & 0.05 & $<$0.01\% & Optimal solution  & Time Limit Exceeded & $<$0.01\% & Optimal solution \\
bental5 & Polyhedral & 0.10 & $<$0.01\% & Optimal solution  & Time Limit Exceeded & $<$0.01\% & Optimal solution \\
bental5 & Polyhedral & 0.15 & $<$0.01\% & Optimal solution  & Time Limit Exceeded & $<$0.01\% & Optimal solution \\
bental5 & Polyhedral & 0.20 & $<$0.01\% & Optimal solution  & Time Limit Exceeded & $<$0.01\% & Optimal solution \\
bental5 & Polyhedral & 0.25 & $<$0.01\% & Optimal solution  & Time Limit Exceeded & $<$0.01\% & Optimal solution \\
bental5 & Polyhedral & 0.30 & $<$0.01\% & Optimal solution  & Time Limit Exceeded & $<$0.01\% & Optimal solution \\
adhya1-4 & Polyhedral & 0.05 & $<$0.01\% & Optimal solution  & Time Limit Exceeded & $<$0.01\% & Optimal solution \\
adhya1-4 & Polyhedral & 0.10 & $<$0.01\% & Optimal solution  & Time Limit Exceeded & $<$0.01\% & Optimal solution \\
adhya1-4 & Polyhedral & 0.15 & $<$0.01\% & Optimal solution  & Time Limit Exceeded & $<$0.01\% & Optimal solution \\
adhya1-4 & Polyhedral & 0.20 & $<$0.01\% & Optimal solution  & Time Limit Exceeded & $<$0.01\% & Optimal solution \\
adhya1-4 & Polyhedral & 0.25 & $<$0.01\% & Optimal solution  & Time Limit Exceeded & $<$0.01\% & Optimal solution \\
adhya1-4 & Polyhedral & 0.30 & $<$0.01\% & Optimal solution  & Time Limit Exceeded & $<$0.01\% & Optimal solution \\
\bottomrule
\end{tabular}
\label{tab:dual_pyros_rsbb}
\end{sidewaystable}

Statistics for RsBB algorithm can be evaluated via the variability of the number of tree nodes explored in Figure \ref{fig:nodes}. As a general trend, for all set types the range of the nodes explored reduces with the increase of the uncertainty set size. As the uncertainty size increases, the corresponding robust cuts become more restraining reducing the feasible region more rigorously. The detailed results for the number of nodes explored across all examined instances is presented in Table \ref{tab:all_nodes}. It can be observed that across the columns of Table \ref{tab:all_nodes}, i.e. for fixed uncertainty set type and size, the number of nodes explored increases with the complexity of the problem. Apart from the outliers depicted in Figure \ref{fig:nodes}, for problems bental5, adhya3 and adhya4 there were instances which required more than 100 nodes to be solved to global robust optimality. The aforementioned problems correspond to the problems of higher complexity across the examined case studies, having the largest number of variables and equations. Taking advantage of more rigorous branching and bounding approaches could result in a faster convergence.

\begin{figure}[H]
    \centering
    \includegraphics[width=0.7\linewidth]{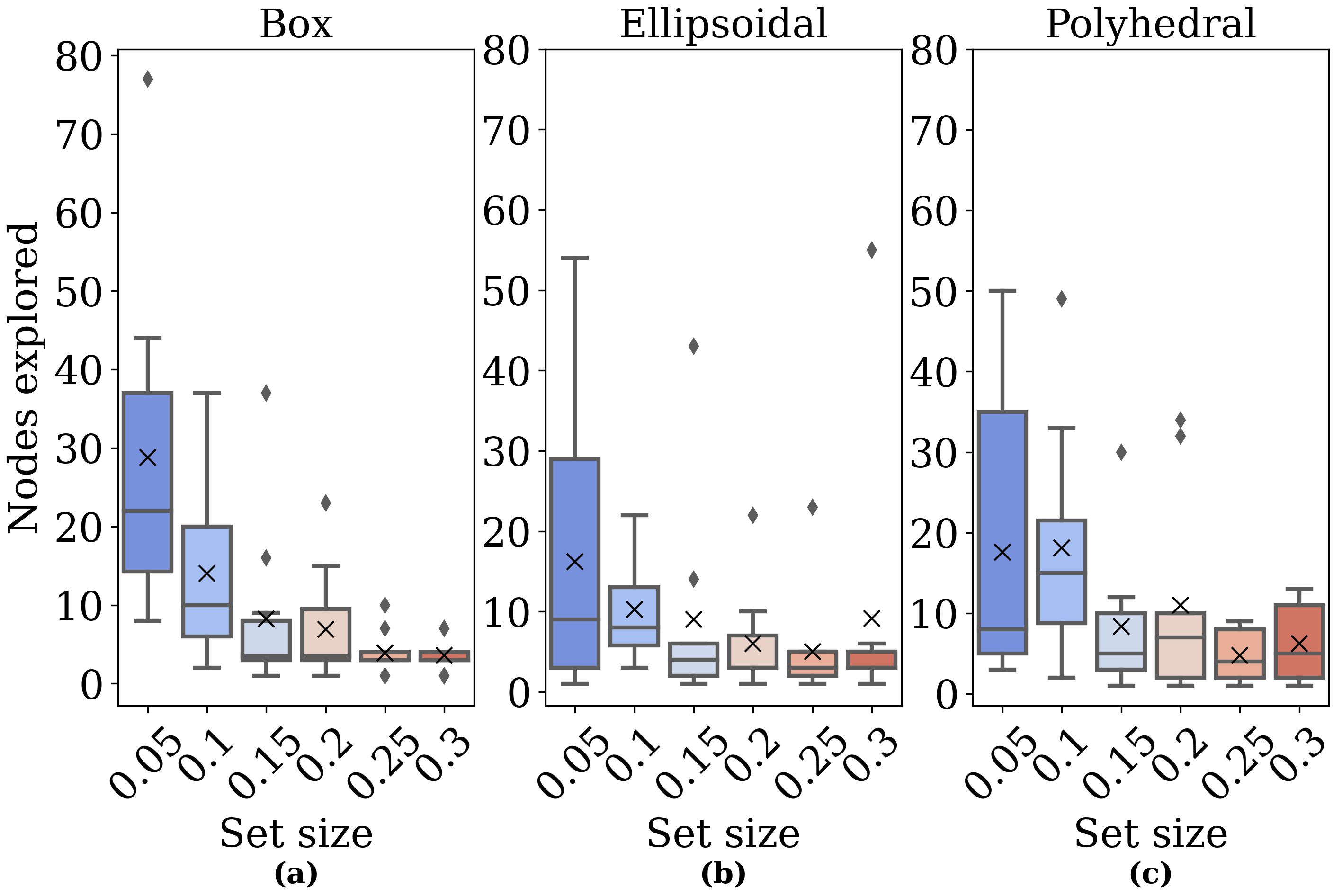}
    \caption{Variability of tree nodes explored for set size and set type for problems solved to robust optimality within 1 hour. }
    \label{fig:nodes}
\end{figure}

\begin{sidewaystable}
\centering
\caption{Number of nodes explored by RsBB across uncertainty sets}
\label{tab:all_nodes}

\begin{tabular}{l*{18}{c}}
\toprule
& \multicolumn{6}{c}{Box ($\Psi$)} & \multicolumn{6}{c}{Ellipsoidal ($\Omega$)} & \multicolumn{6}{c}{Polyhedral ($\Gamma$)} \\
\cmidrule(lr){2-7}\cmidrule(lr){8-13}\cmidrule(lr){14-19}
Problem & 0.05 & 0.10 & 0.15 & 0.20 & 0.25 & 0.30 & 0.05 & 0.10 & 0.15 & 0.20 & 0.25 & 0.30 & 0.05 & 0.10 & 0.15 & 0.20 & 0.25 & 0.30 \\
\midrule
haverly1 & 13 & 3  & 3  & 3  & 3  & 3  & 3  & 3  & 2  & 3  & 3  & 3  & 5  & 7  & 3  & 2  & 2  & 2  \\
haverly2 & 18 & 6  & 3  & 3  & 3  & 3  & 11 & 6  & 2  & 3  & 3  & 3  & 11 & 14 & 6  & 2  & 2  & 2  \\
haverly3 & 37 & 37 & 37 & 23 & 3  & 3  & 5  & 5  & 4  & 2  & 2  & 2  & 5  & 6  & 5  & 4  & 2  & 2  \\
bental4  & 10 & 2  & 1  & 1  & 1  & 1  & 3  & 10 & 6  & 1  & 1  & 1  & 3  & 15 & 5  & 1  & 1  & 1  \\
bental5  & 37 & 20 & 16 & 15 & 10 & 170& 1  & 186& 1009* & 1323* & 1874* & 3235* & 3  & 33 & 185 & 308 & 389 & 395 \\
foulds2  & 8 & 7  & 1 & 1  & 1  & 1  & 9  & 6  & 6 & 22 & 23 & 55 & 8  & 2  & 10 & 10 & 9  & 12 \\
adhya1   & 44 & 11 & 9  & 3  & 4  & 3  & 29 & 22 & 3  & 3  & 2  & 3  & 38 & 23 & 3  & 7  & 8  & 11 \\
adhya2   & 24 & 10 & 5  & 5  & 4  & 4  & 54 & 19 & 1  & 10 & 5  & 5  & 50 & 17 & 1  & 7  & 4  & 8  \\
adhya3   & 20 & 163& 4  & 4  & 4  & 7  & 31 & 110& 43 & 3  & 1  & 6  & 35 & 49 & 12 & 34 & 7  & 5  \\
adhya4   & 77 & 30 & 3  & 11 & 7  & 7  & 412& 11 & 14 & 7  & 5  & 4  & 118& 15 & 30 & 32 & 8  & 13 \\
\bottomrule
\addlinespace
\multicolumn{19}{l}{\footnotesize\emph{* Time limit exceeded.}}\\
\end{tabular}
\end{sidewaystable}

Increasing the resilience of a solution is most often at the expense of optimality \cite{Bertsimas04}. The dependence of the robust optimal objective to the uncertainty set type and size is presented in Figure \ref{fig:obj} as a percentage increase to the nominal objective. For the same uncertainty size, box uncertainty set results in higher cost increase followed by the ellipsoidal and the polyhedral set. Note that for bental5 and ellipsoidal uncertainty set $\Omega=0.3$ the solution obtained after 24h of runtime corresponded to a 15\% gap, hence the \% increase to the nominal was not included in the heatmap. The cost increase is most commonly accompanied by an alternative optimal solution, hence a different BB search route. Haverly3 is the most flexible problem that meets up to a 30\% cost increase even with a 30\% increase of the uncertain parameter. Bental5 and adhya4 have the least operational flexibility since even a 5\% fluctuation in the inlet quality results in a 50\% cost increase. The remaining adhya instances present a sharp cost increase for uncertainty sizes higher than 0.1 stemming from an alternative pooling network operation. Hence the search for an optimal robust solution is not solely dictate by the problem size and uncertainty set characterization, but also by the inherent flexibility of each examined problem. 
\begin{figure}[H]
    \centering
    \includegraphics[width=\linewidth]{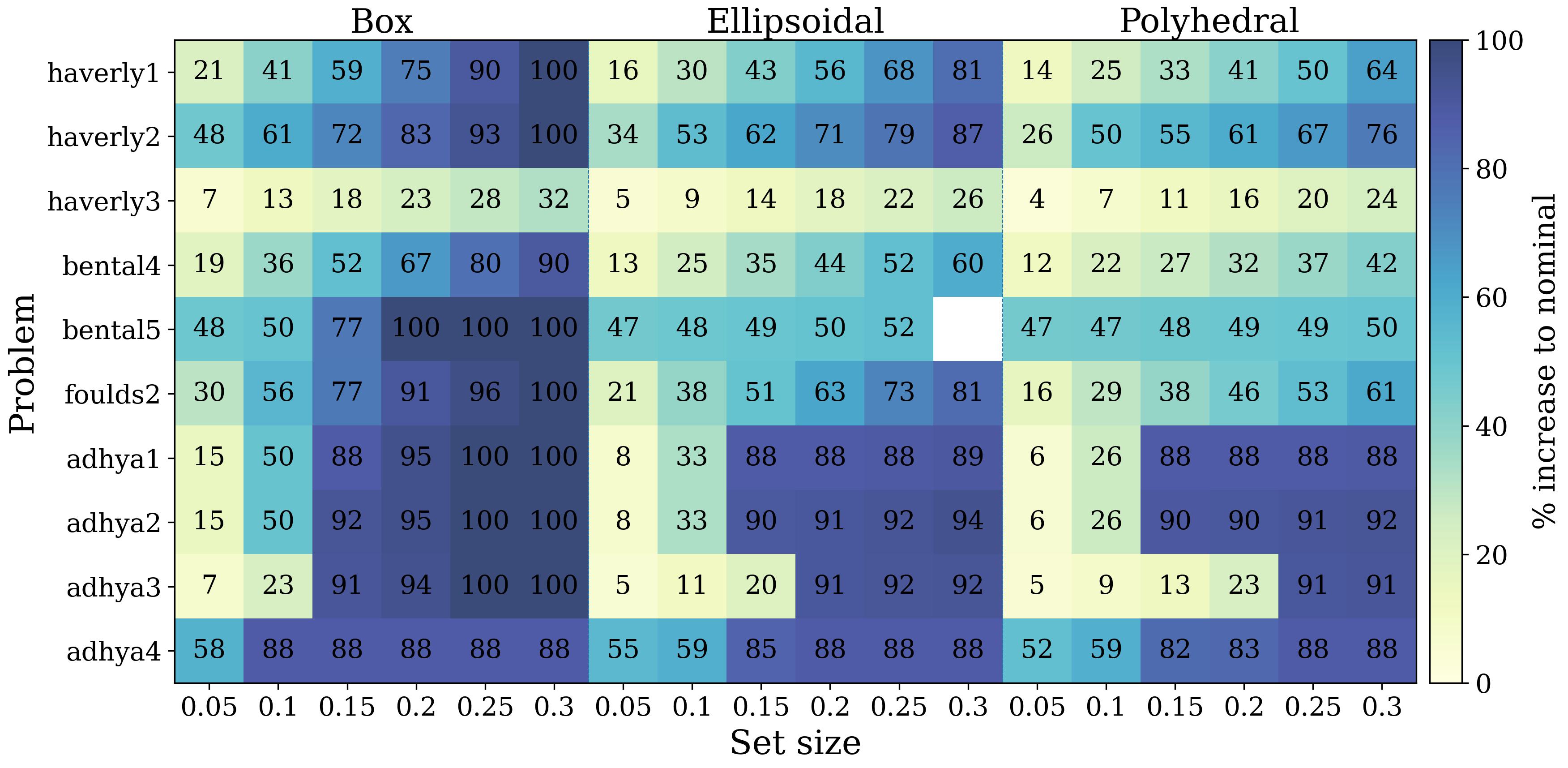}
    \caption{Increase of the robust optimal objective value with respect to the nominal solution. }
    \label{fig:obj}
\end{figure}

Focusing on a specific problem, i.e. foulds2, we can investigate the dependence of the RsBB convergence on the uncertainty set size and set type as in Figure \ref{fig:conv_foulds}, where the relative optimality gap is evaluated as $\frac{|UB-LB|}{|LB|}\times 100$. Increasing the uncertainty size results in a decrease of the RsBB iterations and at the same time a decrease in the optimality gap at the root node. Fixing the same uncertainty size to 0.15, faster convergence is achieved for the box uncertainty set followed by the ellipsoidal and the polyhedral, analogous to the decrease of the set conservativism. For foulds2 with $\Psi=0.15$, RsBB terminates at a robust optimal solution on the root node, giving grounds for further investigation. Figure \ref{fig:cut_foulds} depicts the increase of the objective value in robust cut iterations at the root node of foulds2 from Figure \ref{fig:conv_foulds}. For $\Psi=0.15$, RsBB performs 11 robust cut iterations to obtain the robust optimal solution at the root node. In contrast, for $\Psi=0.1$, fewer cut iterations lead to a higher optimality gap at the root node. For the polyhedral set and $\Gamma=0.15$ only 3 cut iterations take place in the root node, resulting in an initial gap of 9.5\%. It is important to highlight that the objective values retrieved at the root node in \ref{fig:cut_foulds}b correspond to the robust optimal solution and the extra tree nodes are used to tighten the solution of the linear relaxation. For the box and polyhedral uncertainty sets solving the lower-level problem has a finite number of solutions, which are the vertices of the polytopes. Finiteness has also been proved in the case of ellipsoidal set by \cite{casero}, however as displayed in Figure \ref{fig:cut_foulds}b there can be multiple degenerate lower-level violating solutions that do not affect the feasibility region of the upper-level problem.
The optimal objective value of foulds2 for the nominal uncertain parameter is at -1100. The right axis of Figure \ref{fig:cut_foulds} depicts the \% increase of the cost as different uncertainty set and sizes and considered. Allowing a 15\% perturbation with the polyhedral set will result in an increase of 38\% of the objective value, 51\% for the ellipsoidal set and 77\% for the box set.

\begin{figure}[H]
    \centering
    \includegraphics[width=\linewidth]{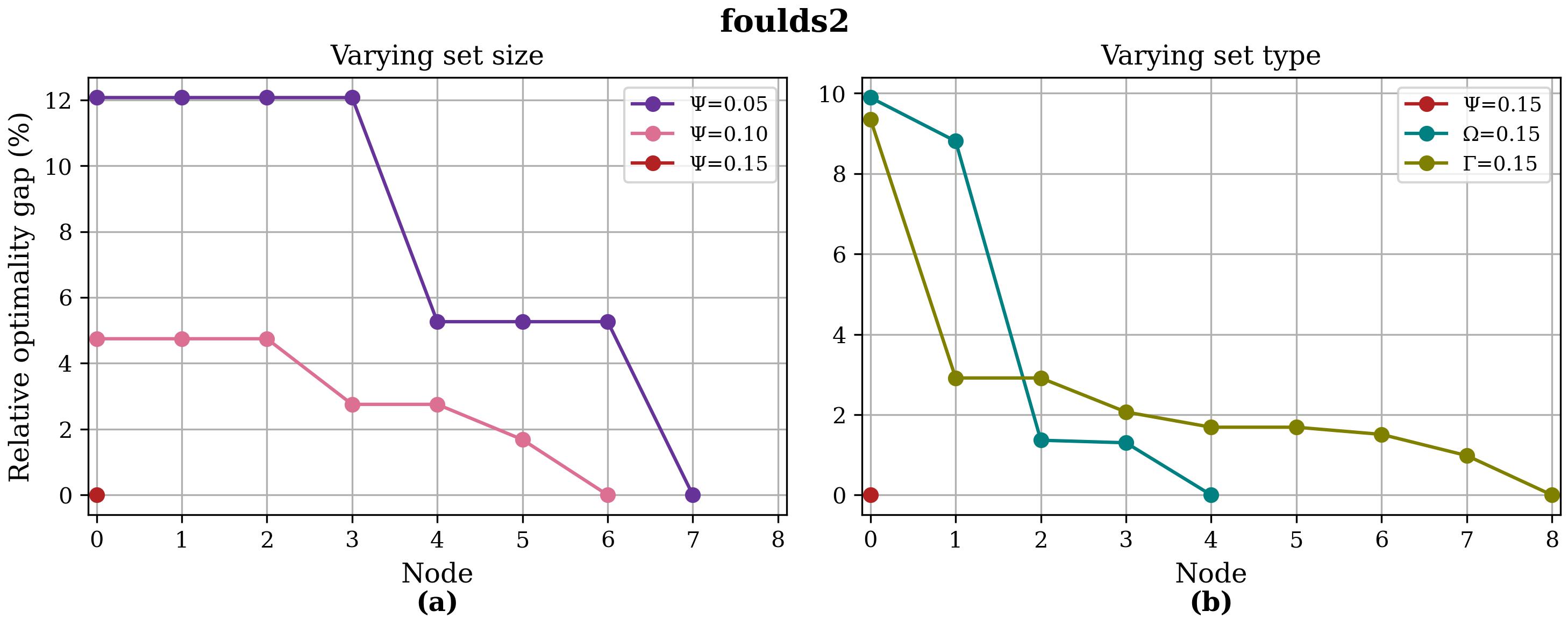}
    \caption{Foulds2 relative optimality gap closure a) for varying size of box uncertainty set, b) and varying set types of size 0.15.}
    \label{fig:conv_foulds}
\end{figure}

\begin{figure}[H]
    \centering
    \includegraphics[width=\linewidth]{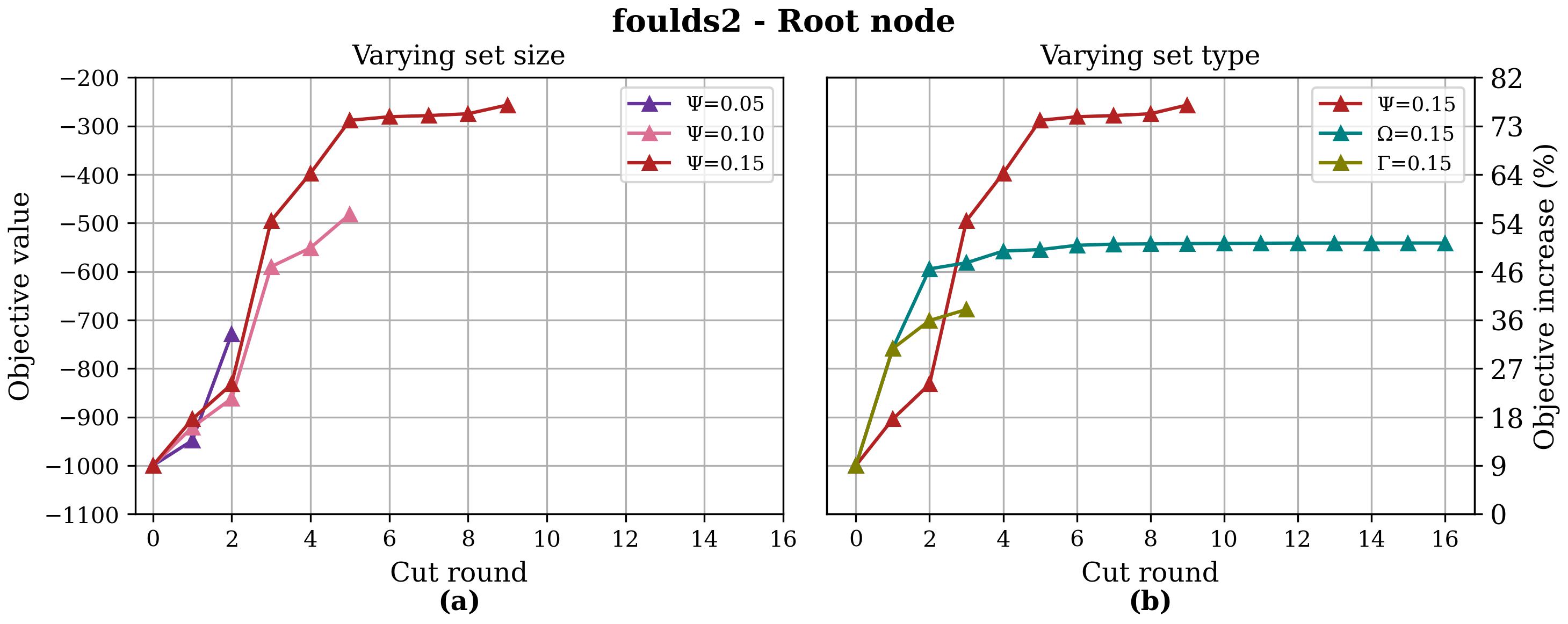}
    \caption{Foulds2 objective value in robust cut iterations at root node a) for varying size of box uncertainty set, b) and varying set types of size 0.15.}
    \label{fig:cut_foulds}
\end{figure}
The detailed bound convergence of the foulds2 instances is presented in Figure \ref{fig:foulds_bounds}. Apart from the root node, it can be observed that an additional robust cut round is required for $\Psi=0.05$ and $\Omega=0.15$ foulds2 instances. In the former, the added cuts are driving the upper-bound  (UB) to convergence. For the ellipsoidal set, the UB has obtained a global solution at the root node and the remaining BB nodes are required to close the gap with the lower-bound (LB). Similarly, for $\Psi=0.10$ and $\Gamma=0.15$ the UB has obtained the robust optimal solution at root node.  The comparison of the bound convergence across problems with different dimensionality is performed in Figure \ref{fig:conv_size}. In both instances an additional cut round is performed in addition to the one at root node which is driving both the UB and LB to convergence. With the second cut round the robust optimal solution is attained for both instances, however the LB convergence is slower resulting in multiple plateau regions. For $\Psi=0.05$ the nodes explored by haverly1 are greater than the nodes explored by foulds2 despite the fact that haverly1 has smaller dimensionality, which can be attributed to the difference of the relative optimality gap at root node. For haverly1 the relative optimality gaps is 87\%, 12\% for foulds2 and 17\% for adhya2. Overall, while the robust cuts result in a rigorous UB convergence, employing tigther lower bounding envelopes could further decrease the RsBB tree size and CPU time. 

\begin{figure}[H]
    \centering
    \includegraphics[width=\linewidth]{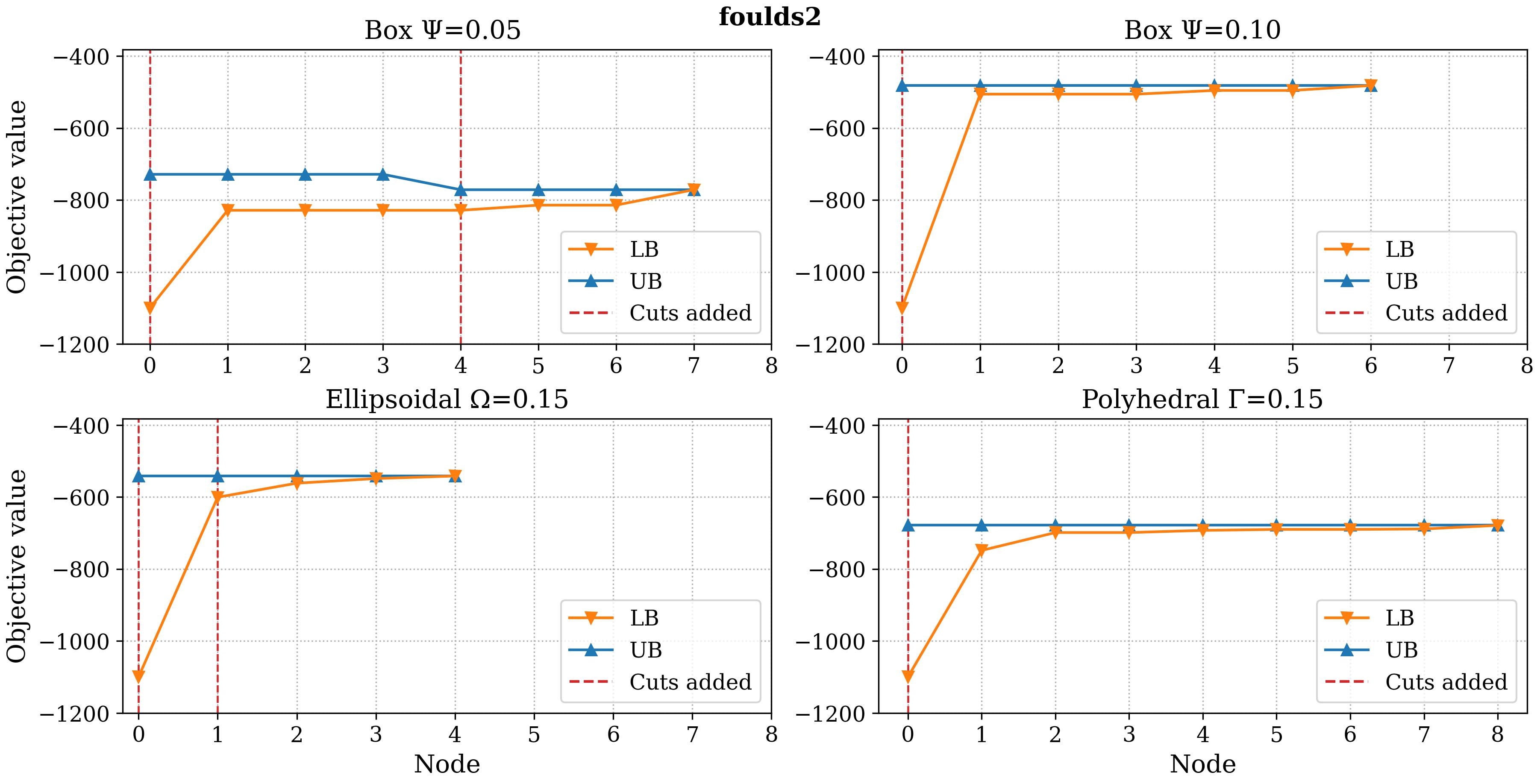}
    \caption{Foulds2 bound convergence for varying uncertainty set and size. }
    \label{fig:foulds_bounds}
\end{figure}

\begin{figure}[H]
    \centering
    \includegraphics[width=\linewidth]{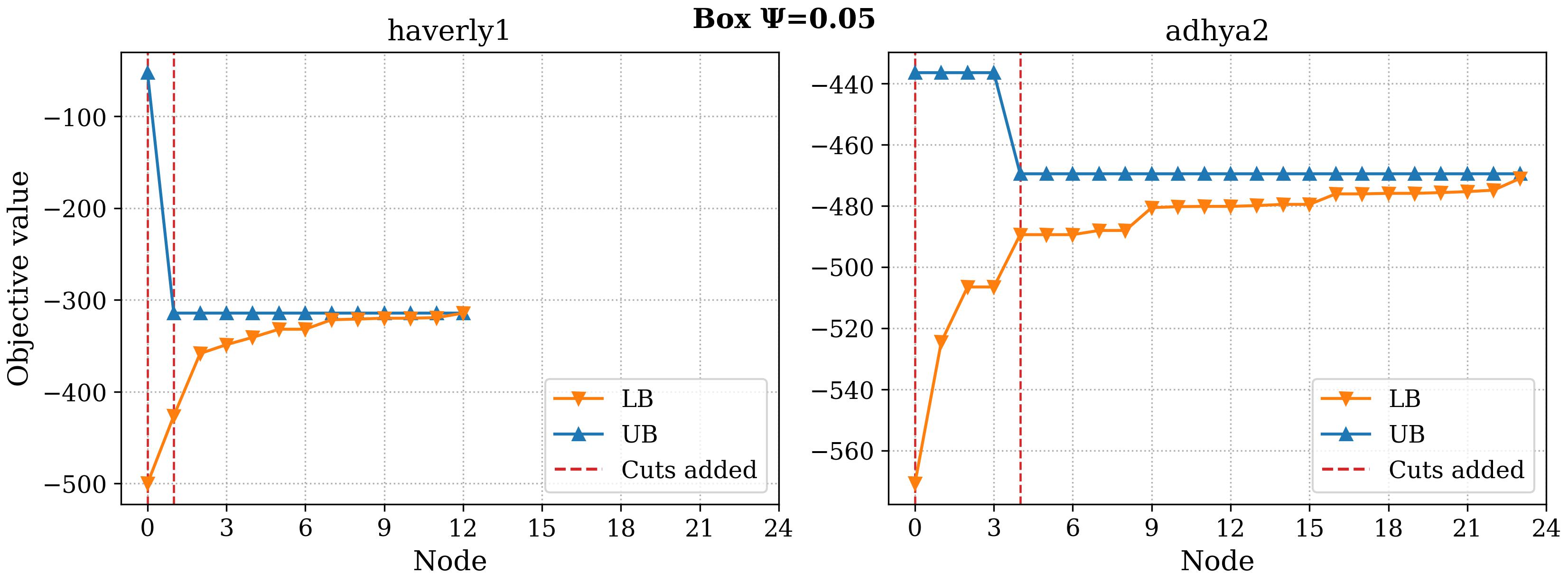}
    \caption{haverly1 and adhya2 bound convergence for box uncertainty $\Psi=0.05$.}
    \label{fig:conv_size}
\end{figure}

Throughout the examined results is has been highlighted that increasing the conservativism of the problem facilitates and accelerates the performance of the RsBB algorithm. Notably, robust cuts serve as a powerful feasibility bounds tightening method, significantly reducing the feasible region for the sBB tree nodes, which enhances the overall efficiency and accuracy of the algorithm.

\section{Conclusion}
In this work, we introduced a novel algorithm that integrates global and robust optimization methods. The proposed RsBB algorithm combines the principles of spatial branch-and-bound and robust cutting planes to solve continuous non-convex problems under convex uncertainty sets. We applied our approach to QCQP pooling problems utilizing McCormick envelopes to obtain convex lower bounds. The inlet quality parameter was treated as uncertain, with its range defined by box, ellipsoidal, and polyhedral uncertainty sets of varying sizes. The performance of the RsBB algorithm was compared to state-of-the-art methods that rely on global solvers. Our computational results demonstrate that the proposed method can solve the 97\% of the tested problems to robust optimality. The computational time was an order of magnitude higher to the direct solve of the dual problem with a global solver. However further improvements via the use of tighter approximations such as RLT constraints and using dynamic branching approaches could further reduce the CPU time of our methodology. We observed that using robust cutting planes in the sBB algorithm serves as an effective feasibility bounds tightening method. Finally, the integration of robust cuts is proven to have a beneficial impact on the sBB algorithm facilitating the optimality search both in terms of problem tractability and CPU time. The current implementation serves as a prototype for the integration of global and robust optimization methodologies. The efficiency of the proposed methodology strongly depends on the tractability of the non-convex problem and the efficiency of the employed relaxations. In the former case, to facilitate the convergence of the local solver, right-hand side restrictions from the SIP literature could be employed to gradually obtain a feasible solution. In the latter case, tailored relaxations could be employed based on SDP relaxations or taking advantage of problem-specific tailored relaxations. Extending the proposed RsBB algorithm to other classes of general non -convex problems is not a trivial task, and it constitutes a promising research direction.

\backmatter

\section*{Nomenclature for the pooling problem}
\subsection*{Sets}
\begin{tabbing}
$i$   \hspace{5em} \=   Feed components\\
$l$      \>Pools\\
$j$      \>Products \\
$k$       \>Qualities 
\end{tabbing}   

\subsection*{Parameters}
\begin{tabbing}
$\xi_k$   \hspace{5em} \=   Random uncertain variable of quality\\ 
$\Psi$  \> Adjustable parameter controlling the size of the uncertainty set\\
$A_i$   \>Upper bound for component availability\\
$D_j$       \>Upper bound for product demand\\
$S_l$       \>Upper bound for pool size \\
$P_{jk}^U$   \>Upper bound for product quality\\
$P_{jk}^L$   \>Lower bound for product quality\\
$C_{ik}$      \>Nominal level of quality for component\\
$\tilde{C}_{ik}$  \>True level of quality for component\\
$\hat{C}_{ik}$      \>Perturbed level of quality for component\\
$c_{i}$       \>Unit price for component\\
$d_{j}$       \>Unit price for product\\
\end{tabbing}  

\subsection*{Variables}
\begin{tabbing}
$q_{il}$ \hspace{5em}   \= Fraction of flow of component $i$ entering pool $l$\\
$y_{lj}$ \> Flow from pool $l$ to product $j$\\
$z_{ij}$ \> Flow from component $i$ to product $j$\\
$f_j$     \> Total flow to product $j$\\
$v_{ilj}$  \> Auxiliary variable for flow from component $i$ passing pool $l$ to product $j$
\end{tabbing}

\bmhead{Data availability statement}
Data and software will be made available by the authors upon reasonable request.
\bmhead{Acknowledgements}
Financial support under the EPSRC grants ADOPT (EP/W003317/1) and RiFtMaP (EP/V034723/1) is gratefully acknowledged by the authors.

\bigskip

\bibliography{GRON}

\end{document}